\documentclass[pdflatex,sn-mathphys-num]{sn-jnl}

\usepackage{graphicx}%
\usepackage{multirow}%
\usepackage{amsmath,amssymb,amsfonts}%
\usepackage{amsthm}%
\usepackage{mathrsfs}%
\usepackage[title]{appendix}%
\usepackage{xcolor}%
\usepackage{textcomp}%
\usepackage{manyfoot}%
\usepackage{booktabs}%
\usepackage{algorithm}%
\usepackage{algorithmicx}%
\usepackage{algpseudocode}%
\usepackage{listings}%

\usepackage{amsthm}
\usepackage{algorithm}
\usepackage{mathtools}
\usepackage{multirow}
\usepackage{amsmath,  enumerate}
\usepackage{array}
\usepackage{caption}
\usepackage{booktabs}
\usepackage{tabularx}
\newtheorem{theorem}{Theorem}[section] %
\newtheorem{lemma}[theorem]{Lemma} %
\newtheorem{proposition}[theorem]{Proposition} %
\theoremstyle{definition} %
\newtheorem{definition}{Definition}[section] %
\newtheorem{assumption}{Assumption} %
\theoremstyle{remark} %
\newtheorem*{remark}{Remark} %

\usepackage{float}
\usepackage{array}
\usepackage{graphicx,xcolor}
\usepackage{subfigure}
\usepackage{multirow}

\newcommand{\grad}{\mathrm{grad}}
\newcommand{\dist}{\mathrm{dist}}

\newcommand{\norm}[1]{\left\|#1\right\|}

\usepackage{setspace}

\begin{document}

\title[Article Title]{Riemannian conditional gradient methods for composite optimization problems}


\author*[1]{\fnm{Kangming} \sur{Chen}}\email{kangming@amp.i.kyoto-u.ac.jp}

\author[1]{\fnm{Ellen H.} \sur{Fukuda}}\email{ellen@i.kyoto-u.ac.jp}

\affil*[1]{\orgdiv{Graduate School of Informatics}, \orgname{Kyoto University}, \orgaddress{ \city{Kyoto}, \postcode{606-8501},  \country{Japan}}}




\abstract{In this paper, we propose Riemannian conditional gradient methods for minimizing composite functions, i.e., those that can be expressed as the sum of a smooth function and a retraction-based convex function. We analyze the convergence of the proposed algorithms, utilizing three types of step-size strategies: adaptive, diminishing, and those based on the Armijo condition. 
We establish the convergence rate of \(\mathcal{O}(1/k)\) for the adaptive and diminishing step sizes, where \(k\) denotes the number of iterations. Additionally, we derive an iteration complexity of \(\mathcal{O}(1/\epsilon^2)\) for the Armijo step-size strategy to achieve \(\epsilon\)-optimality, where \(\epsilon\) is the optimality tolerance.
Finally, the effectiveness of our algorithms is validated through some numerical experiments performed on the sphere and  Stiefel manifolds.}





\maketitle

\section{Introduction}

The conditional gradient (CG) method, also known as the Frank-Wolfe algorithm, is a first-order iterative method widely used for solving constrained nonlinear optimization problems. Developed by Frank and Wolfe in 1956~\cite{frank1956algorithm}, the CG method focuses on minimizing a convex function over a convex constraint set, particularly when the constraint set is difficult, in the sense that the projection onto it is computationally demanding.

The convergence rate of  $\mathcal{O}(1/k)$, where $k$ denotes the number of iterations, for the conditional gradient method has been presented in several studies~\cite{bach2015duality, pmlr-v28-jaggi13}. Recent advancements in the field have focused on accelerating the method~\cite{li2021momentum, zhang2021accelerating}. The versatility of this method is also clear, since many applications can be found in the literature (see, for instance, \cite{lacoste2013block, dalmasso2023efficient}). Furthermore, the CG method has recently been extended to multiobjective optimization~\cite{assunccao2021conditional,assunccao2023generalized,GF24}.

Meanwhile, composite optimization problems, where the objective function can be decomposed into the sum of two functions, have become increasingly important in modern optimization theory and applications. Specifically, the objectives of these problems take the form \( F(x) := f(x) + g(x) \), where \( f \) is a smooth function and \( g \) is a possibly non-smooth and convex function. This structure arises naturally in many areas, including machine learning and signal processing, where \( f \) often represents a data fidelity term and \( g \) serves as a regularization term to promote desired properties such as sparsity or low rank~\cite{wright2009sparse, tibshirani1996regression, combettes2005signal}. 

To deal with such problems efficiently, various methods have been developed. Classical approaches include proximal gradient methods~\cite{beck2017first, FM81,parikh2014proximal}, which address the non-smooth term \( g \) using proximal operators, and accelerated variants such as FISTA~\cite{beck2009fast}.  Other methods, such as the alternating direction method of multipliers (ADMM) \cite{boyd2011distributed}, decompose the problem further to enable parallel and distributed computation. For problems with specific structures or manifold constraints, methods like the Riemannian proximal gradient have been proposed in~\cite{chen2020proximal, huang_riemannian_2022}, extending proximal techniques to non-Euclidean settings. Recently, the CG method in Riemannian manifolds has been proposed in~\cite{weber_riemannian_2023}, and the CG method for composite functions, called the generalized conditional gradient (GCG) method, has also been discussed in many works~\cite{assunccao2023generalized, yu2017generalized, zhao2023analysis, cheung2015efficient}.

In this paper, we propose GCG methods on Riemannian manifolds, with general retractions and vector transport. We also focus on three types of step-size strategies: adaptive, diminishing, and those based on the Armijo condition. We provide a detailed discussion on the implementation of subproblem solutions for each strategy. For each step-size approach, we analyze the convergence of the method. 
Specifically, for the adaptive and diminishing step-size strategies, we establish a convergence rate of  $\mathcal{O}(1/k)$, where \(k\) represents the number of iterations. For the Armijo step size strategy, we derive an iteration complexity of \(\mathcal{O}(1/\epsilon^2)\), where \(\epsilon\) represents the desired accuracy in the optimality condition. 

The outline of this paper is as follows.
In section~\ref{sec: preliminaries}, we provide an overview of fundamental concepts and results in Riemannian optimization.
In Section~\ref{sec3}, we introduce the generalized conditional gradient method on Riemannian manifolds with three different types of step size. 
Section~\ref{sec4} presents the convergence analysis, while Section~\ref{sec: subproblem} focuses on the discussion of the subproblem.
Section~\ref{sec: accelerate} explores the accelerated version of the proposed method. 
Numerical experiments are presented in Section~\ref{sec: expreriment} to illustrate the effectiveness of our approach. Finally, in Section~\ref{sec: conclusion}, we conclude the paper and suggest potential directions for future research.

\section{Preliminaries}\label{sec: preliminaries}

This section summarizes essential definitions, results, and concepts fundamental to Riemannian optimization~\cite{sato_riemannian_2021}.
The transpose of a matrix $A$ is denoted by $A^T$, and the identity matrix with dimension $\ell$ is written as $I_\ell$.
A Riemannian manifold \(\mathcal{M}\) is a smooth manifold equipped with a Riemannian metric, which is a smoothly varying inner product \(\langle \eta_x, \sigma_x \rangle_x \in \mathbb{R}\), defined on the tangent space at each point \(x \in \mathcal{M}\), where \(\eta_x\) and \(\sigma_x\) are tangent vectors in the tangent space \(T_x \mathcal{M}\).
The tangent space at a point \(x\) on the manifold \(\mathcal{M}\) is denoted by \(T_x \mathcal{M}\), while the tangent bundle of \(\mathcal{M}\) is denoted as \(T\mathcal{M} := \{(x, d) \mid d \in T_x \mathcal{M}, x \in \mathcal{M}\}\).
The norm of a tangent vector \(\eta \in T_x \mathcal{M}\) is defined as \(\|\eta\|_x := \sqrt{\langle\eta, \eta\rangle_x}\).
When the norm and the inner product are written as \( \|\cdot\| \) and $\langle \cdot, \cdot \rangle$,
without the subscript, then they are the Euclidean ones.
For a map \(F \colon \mathcal{M} \rightarrow \mathcal{N}\) between two manifolds \(\mathcal{M}\) and \(\mathcal{N}\), the derivative of \(F\) at a point \(x \in \mathcal{M}\), denoted by \(\mathrm{D} F(x)\), is a linear map \(\mathrm{D} F(x) \colon T_x \mathcal{M} \rightarrow T_{F(x)} \mathcal{N}\) that maps tangent vectors from the tangent space of \(\mathcal{M}\) at \(x\) to the tangent space of \(\mathcal{N}\) at \(F(x)\).
The Riemannian gradient \(\operatorname{grad} f(x)\) of a smooth function \(f \colon \mathcal{M} \rightarrow \mathbb{R}\) at a point \(x \in \mathcal{M}\) is defined as the unique tangent vector at \(x\) that satisfies the equation \(\langle\operatorname{grad} f(x), \eta\rangle_x = \mathrm{D} f(x)[\eta]\) for every tangent vector \(\eta \in T_x \mathcal{M}\).
We also define the Whitney sum of the tangent bundle as \(T \mathcal{M} \oplus T \mathcal{M} := \left\{(\xi, d) \mid \xi, d \in T_x \mathcal{M}, x \in \mathcal{M}\right\}\).

A retraction is a mapping that maps points from the tangent space of a manifold back onto the manifold.   
\begin{definition}\cite{sato_riemannian_2021}
    A smooth map \( R \colon T\mathcal{M} \to \mathcal{M} \) is called a retraction on a smooth manifold \(\mathcal{M}\) if its restriction to the tangent space \( T_x\mathcal{M} \) at any point \( x \in \mathcal{M} \), denoted by \( R_x \), satisfies the following conditions:
    \begin{enumerate}
    \item $R_x\left(0_x\right)=x$,
    \item $\mathrm{D} R_x\left(0_x\right)=\operatorname{id}_{T_x \mathcal{M}}$ for all $x \in \mathcal{M}$, 
    \end{enumerate}
    where \(0_x\) denotes the zero vector in \(T_x \mathcal{M}\), and \(\operatorname{id}_{T_x \mathcal{M}}\) represents the identity map in \(T_x \mathcal{M}\).
\end{definition}
Depending on the problem and the properties of the manifold, different types of retractions can be utilized. Many Riemannian optimization algorithms employ a retraction that generalizes the exponential map on \(\mathcal{M}\).
For instance, consider an iterative method that generates iterates $\{x^k\}$. In Euclidean spaces, the update takes the form $x^{k+1}=x^k + t_k d ^k$, where $d^k$ is a descent direction and $t_k$ is a step size. In the Riemannian case, this update is generalized to:
$$x^{k+1}=R_{x^k}(t_k d^k),\quad \text{for }  k = 0,1,2, \ldots.$$

Another essential concept is vector transport, which is critical in Riemannian optimization to maintain the intrinsic geometry of the manifold during computations.
\begin{definition}\cite{sato_riemannian_2021}
A map $\mathcal{T} \colon T \mathcal{M} \oplus T \mathcal{M} \rightarrow T \mathcal{M}$ is called a vector transport on $\mathcal{M}$ if it satisfies the following conditions:
\begin{enumerate}
\item There exists a retraction $R$ on $\mathcal{M}$ such that $\mathcal{T}_d(\xi) \in T_{R_x(d)} \mathcal{M}$ for all $x \in \mathcal{M}$ and $\xi, d \in T_x \mathcal{M}$.
\item For any $x \in \mathcal{M}$ and $\xi \in T_x \mathcal{M}, \mathcal{T}_{0_x}(\xi)=\xi$ holds, where $0_x$ is the zero vector in $T_x \mathcal{M}$, i.e., $\mathcal{T}_{0_x}$ is the identity map in $T_x \mathcal{M}$.
\item For any $a, b \in \mathbb{R}, x \in \mathcal{M}$, and $\xi, d, \zeta \in T_x \mathcal{M}, \mathcal{T}_d(a \xi+b \zeta)=a \mathcal{T}_d(\xi)+b \mathcal{T}_d(\zeta)$ holds, i.e., $\mathcal{T}_d$ is a linear map from $T_x \mathcal{M}$ to $T_{R_x(d)} \mathcal{M}$.
\end{enumerate}
\end{definition}

The adjoint operator of a vector transport \(\mathcal{T}\), denoted by \(\mathcal{T}^{\sharp}\), satisfies:  
\[
\langle \xi_y, \mathcal{T}_{\eta_x} \zeta_x \rangle_y = \langle \mathcal{T}_{\eta_x}^{\sharp} \xi_y, \zeta_x \rangle_x, \quad \forall \eta_x, \zeta_x \in T_x \mathcal{M}, \, \xi_y \in T_y \mathcal{M},
\]  
where \( y = R_x(\eta_x) \).  
The inverse operator of \(\mathcal{T}\), denoted by \(\mathcal{T}_{\eta_x}^{-1}\), satisfies:  
\[
\mathcal{T}_{\eta_x}^{-1} \mathcal{T}_{\eta_x} = \text{id}, \quad \forall \eta_x \in T_x \mathcal{M},
\]  
where \(\text{id}\) is the identity operator.

Note that the map \(\mathcal{T}\), defined as \(\mathcal{T}_d(\xi) := \mathrm{P}_{\gamma_{x,d}}^{1 \leftarrow 0}(\xi)\), is also a vector transport. Here, \(\mathrm{P}_{\gamma_{x,d}}^{1 \leftarrow 0}\) denotes the parallel transport along the geodesic \(\gamma_{x,d}(t) := \operatorname{Exp}_x(t \, d)\), which connects \(\gamma_{x,d}(0) = x\) and \(\gamma_{x,d}(1) = \operatorname{Exp}_x(d)\). In this context, the exponential map \(\operatorname{Exp}\) acts as the retraction.
Parallel transport is a specific case of vector transport that preserves the length and angle of vectors as they are transported along geodesics, thereby preserving the intrinsic Riemannian metric of the manifold exactly. In contrast, vector transport is a more general operation that allows for the movement of vectors between tangent spaces in a way that approximately preserves geometric properties, and it may not necessarily follow geodesics or maintain exact parallelism.
If there is a unique geodesic between any two points in \(\mathcal{X} \subset \mathcal{M}\), then the exponential map has an inverse \(\operatorname{Exp}_x^{-1} \colon \mathcal{X} \rightarrow T_x \mathcal{M}\). This geodesic represents the unique shortest path and the geodesic distance between \(x\) and \(y\) in \(\mathcal{X}\) is given by \(\left\|\operatorname{Exp}_x^{-1}(y)\right\|_x = \left\|\operatorname{Exp}_y^{-1}(x)\right\|_y\).
The Riemannian distance on a connected Riemannian manifold $\mathcal{M}$ is defined as
$
\text { dist : } \mathcal{M} \times \mathcal{M} \rightarrow \mathbb{R}: \operatorname{dist}(x, y)=\inf _{\Gamma} L(\gamma)
$, where $\Gamma$ is the set of all curves in $\mathcal{M}$ joining points $x$ and $y$. 
For any two points on a geodesically complete Riemannian manifold, the Riemannian distance is precisely the geodesic distance.

Now, we define geodesic convexity sets and functions, \(L\)-smoothness, and related concepts on Riemannian manifolds as follows:

\begin{definition}\cite{Boumal_2023}
    A set $\mathcal{X}$ is called geodesically convex if for any $x, y \in \mathcal{X}$, there is a geodesic $\gamma$ with $\gamma(0)=x, \gamma(1)=y$ and $\gamma(t) \in \mathcal{X}$ for all $t \in[0,1]$.
\end{definition}

\begin{definition}
    A function $h \colon \mathcal{M} \rightarrow \mathbb{R}$ is called geodesically convex, where $\mathcal{S}\subseteq \mathcal{M}$ is a geodesically convex set, if for any $p, q \in \mathcal{M}$, we have $h(\gamma(t)) \leq(1-t) h(p)+t h(q)$ for any $t \in[0,1]$, where $\gamma$ is the geodesic connecting $p$ and $q$.
\end{definition}

\begin{proposition}\cite{weber_riemannian_2023}
    Let $h \colon \mathcal{M} \rightarrow \mathbb{R}$ be a smooth and geodesically convex function. Then, for any $x, y \in \mathcal{M}$,
    we have
    $$
    h(y)-h(x) \geq\left\langle\operatorname{grad} h(x), \operatorname{Exp}_x^{-1}(y)\right\rangle_x .
    $$
\end{proposition}
To ensure that the convexity property of functions is preserved when using a general retraction, we make the following definition.
\begin{assumption}\label{assumption convex}
     Let \( h: \mathcal{M} \to \mathbb{R} \) be a smooth function, where \(\mathcal{M}\) is a Riemannian manifold. Consider a retraction \( R_x: T_x \mathcal{M} \to \mathcal{M} \), for any \( x, y \in \mathcal{M} \) such that \( R_x(y) \) and \( R_x^{-1}(y) \) are well-defined, the following inequality holds for all \( \lambda \in [0,1] \):
    \[
    h(R_x(\lambda R_x^{-1}(y))) \leq (1-\lambda) h(x) + \lambda h(y).
    \] 
\end{assumption}   
To distinguish it from geodesic convexity, we define a function $h$ as $R$-convex if it satisfies the conditions outlined in Assumption~\ref{assumption convex}.
Note that the exponential map satisfies this assumption. 
Note that not all retractions satisfy $R$-convexity and it often depends on the specific function (for example, for the Euclidean retraction on the SPD manifold, the function $h(P) = -\log \det(P)$ is globally $R$-convex). However, since a second-order retraction~\cite[Definition 5.42]{Boumal_2023}
approximates the geodesic more closely, a well-behaved function (e.g., a geodesically convex function) is more likely to satisfy $R$-convexity with respect to such a retraction, particularly in a local neighborhood. 
In the following discussion and throughout the development of our algorithms, we restrict our attention to retractions that satisfy this assumption.
Motivated by this assumption, we formally introduce the definition of an $R$-convex set.
\begin{definition}
Let $\mathcal{M}$ be a Riemannian manifold equipped with a retraction $R$. A subset $\mathcal{S} \subseteq \mathcal{M}$ is said to be $R$-convex with respect to $R$ if for any $x, y \in \mathcal{S}$ for which the inverse retraction $R_x^{-1}(y)$ is well-defined, the retraction curve joining them is contained entirely within $\mathcal{S}$. 
Formally, for all $t \in [0, 1]$, the following condition must hold:
\[
R_x\left(t R_x^{-1}(y)\right) \in \mathcal{S}.
\]
\end{definition}

Here, we give the definition of $L$-smoothness, which is also referred to as $L$-retraction-smooth as discussed in~\cite{huang_riemannian_2022}.
\begin{definition}
 A function $h \colon \mathcal{M} \rightarrow \mathbb{R}$ is called  $L$-smooth if there exists $L$ such that
\begin{equation}\label{ineq l-smooth}
    h(y) \leq h(x)+\left\langle\grad h(x), R^{-1}_{x}(y)\right\rangle_x +\frac{L}{2}\left\|{R}_x^{-1}(y)\right\|^2_x \quad \forall x, y \in \mathcal{M}.
\end{equation}
where $\mathcal{T}$ is the vector transport from $T_x \mathcal{M}$ to $T_y \mathcal{M}$. 
\end{definition}
A strong version of \eqref{ineq l-smooth}, as given below, is also used in \cite[Assumption 2.6]{boumal2019global}.
$$
|h(y) - h(x)-\left\langle\grad h(x), R^{-1}_{x}(y)\right\rangle_x |\leq \frac{L}{2}\left\|{R}_x^{-1}(y)\right\|^2_x \quad \forall x, y \in \mathcal{M}.
$$
In addition, if the retraction is chosen to be the exponential map, it follows that $h$ is geodesically $L$-smooth, as defined in~\cite{weber_riemannian_2023} and~\cite[Section 10.4]{Boumal_2023}.

\section{Riemannian generalized conditional gradient method}\label{sec3}

In this paper, we consider the following Riemannian optimization problem:
    \begin{equation}
    \label{eq:prob}
    \begin{aligned}
        \min & \quad F(x):= f(x)+ g(x) \\
        \mbox{s.t.} & \quad x \in \mathcal{X} ,
    \end{aligned}
    \end{equation}
where $\mathcal{X} \subseteq \mathcal{M}$ is a compact $R$-convex set and $F \colon \mathcal{M} \rightarrow \mathbb{R} $ is a composite function with $f\colon \mathcal{M} \rightarrow \mathbb{R}$ being continuously differentiable and $g\colon \mathcal{M} \rightarrow \mathbb{R}$ 
being a closed, $R$-convex and lower semicontinuous (possibly nonsmooth) function.
Here, we can also consider \( F \) as an extended-valued function. All the analyses would still apply, but we omit the infinity for simplicity.
Also, we make the following assumption.
\begin{assumption}\label{assum1}
    The set $\mathcal{X} \subseteq \mathcal{M}$ is an $R$-convex  and complete set where the retraction ${R}$ and its inverse ${R}^{-1}$ are well-defined and smooth over  $\mathcal{X}$. Additionally, the inverse retraction \( R^{-1}_x(y) \) is continuous with respect to both \( x \) and \( y \).
\end{assumption}

As discussed in \cite[Section 10.2]{Boumal_2023}, 
for general retractions, over some domain of \( \mathcal{M} \times \mathcal{M} \) which contains all pairs \( (x, x) \), the map 
$(x, y) \mapsto R_x^{-1}(y)$
can be defined smoothly jointly in \( x \) and~\( y \). In this sense, we can say that the above assumption is reasonable.

Now, let us recall the CG method in the Euclidean space. Consider minimizing a convex function \( \tilde{f} \colon \mathbb{R}^n \rightarrow \mathbb{R} \) over a convex compact set \( \mathcal{X} \subseteq  \mathbb{R}^n\).
At each iteration \( k \), the method solves a linear approximation of the original problem by finding a direction \( s^k \) that minimizes the linearized objective function over the constraint set \( \mathcal{X} \), i.e.,
$s^k = \arg\min_{s \in \mathcal{X}} \langle \nabla \tilde{f}(x^k), s \rangle .$
Then, it updates the current point \( x^k \) using a step size \( \lambda_k \) along the direction \( s^k - x^k \), i.e.,
$x^{k+1} = x^k + \lambda_k (s^k - x^k).$
For composite functions of the form \( \tilde{f}(x) = \tilde{h}(x) + \tilde{g}(x) \), where \( \tilde{h} \) is smooth and \( \tilde{g} \) is convex, the CG method can be adapted to handle the non-smooth term \( \tilde{g} \). The update step becomes:
\[ 
x^{k+1} = \arg\min_{s \in \mathcal{X}} \langle \nabla \tilde{h}(x^k), s \rangle + \tilde{g}(s).
\]
In this case, \( s \mapsto \langle \nabla \tilde{h}(x^k), s \rangle + \tilde{g}(s) \) is the objective function for the subproblem, which remains convex since it is the sum of the convex function  \( \tilde{g} \) and the linear term \( \langle \nabla \tilde{h}(x^k), \cdot \rangle \). Moreover, the boundedness of $\mathcal{X}$ guarantees the existence of a solution to this subproblem.
The step size \( \lambda_k \) can be chosen using various strategies. The CG method is particularly useful when the constraint set \( \mathcal{X} \) is not an easy set, making projection-based methods computationally expensive.

The exponential map, the parallel transport, and specific step size are used in the Riemannian CG method proposed in~\cite{weber_riemannian_2023}. Now we extend this approach to a more general framework for composite optimization problems, without specifying the retraction and the vector transport.
\begin{algorithm}[H]
    \caption{Riemannian generalized conditional gradient method (RGCG)}\label{alg:RGCG}
    
    \textbf{Step 0. Initialization:}
    
    \qquad Choose an initial point \(x^0 \in \mathcal{X}\) and set \(k = 0\).
     
    \textbf{Step 1. Compute the search direction:}
     
    \qquad Compute an optimal solution $p(x^k)$ and the optimal value $\theta(x^k)$ as
    \begin{equation}\label{eq pk}
     p(x^k) = \arg\min_{u \in \mathcal{X}}  \left\langle\grad f(x^k), R^{-1}_{x^k}(u)\right\rangle_{x^k}  +g(u) -g(x^k),
    \end{equation}
    \begin{equation}\label{def theta}
    \theta(x^k) =  \left\langle\grad f(x^k), R^{-1}_{x^k}(p(x^k)) \right\rangle_{x^k} +g(p(x^k)) -g(x^k).
    \end{equation} 
    \qquad Define the search direction by $d(x^k) = R^{-1}_{x^k}(p(x^k))$.\\
    \textbf{Step 2. Compute the step size:} 
    
    \qquad Compute the step size \( \lambda_k \). 

    \textbf{Step 3. Update the iterates:}
     
    \qquad Update the current iterate
     $x^{k+1} = R_{x^k} ( \lambda_k d(x^k)).$
       
    \textbf{Step 4. Convergence check:} 
     
    \qquad If a convergence criteria is met, stop; otherwise, set \( k = k + 1 \) and return to Step~1.
\end{algorithm}
Our method is presented in Algorithm~\ref{alg:RGCG}. In Step~1, we need to solve the subproblem~\eqref{eq pk}, which, similarly to the Euclidean case, is related to the linear approximation of the objective functions. In Step~2, we compute the step size based on some strategies that we will mention later. In Step~3, we update the iterate by using some retraction. Note that if \( u = x^k \) in \eqref{eq pk}, then \( \left\langle \grad f(x^k), R^{-1}_{x^k}(u) \right\rangle_{x^k} + g(u) - g(x^k) = 0 \), which implies that \(\theta(x^k) \leq 0\). And since $d(x^k) = R^{-1}_{x^k}(p(x^k))$, the subproblem is equivalent to
    \begin{equation}\label{eq dk}
         d(x^k) = \arg\min_{d \in T_{x^k}\mathcal{X}}  \left\langle\grad f(x^k), d\right\rangle_{x^k}  +g(R_{x^k}(d)) -g(x^k).
    \end{equation}

We now consider three distinct, rigorously defined strategies, which are also discussed in~\cite{assunccao2021conditional} and elaborated below. 
~\\   

\noindent\textbf{Armijo step size:} Let \( \zeta \in (0,1) \) and \( 0 < \omega_1 < \omega_2 < 1 \). The step size \( \lambda_k \) is determined using the following algorithm:

Step 0. Set $\lambda_{k_0}=1$ and initialize $\ell \leftarrow 0$.

Step 1. If
$$
F(R_{x^k} ( \lambda_k d(x^k))) \leq F(x^k)+\zeta \lambda_{k_{\ell}} \theta (x^k),
$$

\hspace{28pt}  then set $\lambda_k:=\lambda_{k_{\ell}}$ and return to the main algorithm.

Step 2. Find $\lambda_{k_{\ell+1}} \in\left[\omega_1 \lambda_{k_{\ell}}, \omega_2 \lambda_{k_{\ell}}\right]$, update $\ell \leftarrow \ell+1$, and return to Step 1.

\noindent\textbf{Adaptive step size:} For all $k$, set
$$
\lambda_k:=\min \left\{1, \frac{-\theta(x^k)}{L \dist^2(p\left(x^k\right), x^k)}\right\}
=\operatorname{argmin}_{\lambda \in(0,1]}\left\{ \lambda\theta(x^k)+\frac{L}{2} \lambda^2 \dist^2(p(x^k), x^k) \right\} .
$$

\noindent\textbf{Diminishing step size:} For all $k$, set
$$
\lambda_k:=\frac{2}{k+2} .
$$

\section{Convergence analysis}\label{sec4} 
In this section, we will provide the convergence analysis for three types of step sizes. The following assumptions will be used in the convergence results.

\begin{assumption}\label{assu Lsmooth}
    The function $f$ is $L$-smooth, namely, there exist a constant  $L$ such that \eqref{ineq l-smooth} holds.
\end{assumption}

Since $\mathcal{X}$ is a compact set, we can, without loss of generality, assume that $\mathcal{X} =  \mathrm{dom}(g) $ is a compact set, its diameter is finite and can be defined by 
\begin{equation}
    \operatorname{diam}(\mathcal{X}):=\sup _{x, y \in \mathcal{X}} \dist(x, y),
\end{equation}
where $\dist(x, y)$ is defined in Section~\ref{sec: preliminaries}.

\begin{assumption}\label{assu:omega}
    For $\{x^k\}$ generated by Algorithm~\ref{alg:RGCG}, there exists $\Omega>0$ such that $\Omega \geq$ $\max \big\{\max \big\{ \|R_{x^i}^{-1} (x^j)\|_{x^i}: {x^i}, {x^j} \in \{x^k\} \big\}, \operatorname{diam}(\mathcal{X})\big\}$.
\end{assumption}

As \(k\) tends to infinity, we can assume without loss of generality that \(x^k\) lies within a compact neighborhood of \(x^*\). Due to the smoothness of the inverse retraction, this ensures that \(\|R_{x^i}^{-1}(x^j)\|_{x^i}\) remains bounded. Hence, Assumption \ref{assu:omega} is reasonable.
\begin{lemma}\label{lemma:des}
    Let $\{x^k\}$ be generated by Algorithm~\ref{alg:RGCG}. Define $z^k = p(x^k) $ and  recall that $x^{k+1}= R_{x^k}(\lambda R^{-1}_{x^k}(z^k))$.
    Then,
    we have
    \begin{equation}
        F(x^{k+1}) \leq F(x^k)+ \lambda\theta (x^k) +\frac{L\lambda^2}{2}\Omega^2.
    \end{equation} 
\end{lemma}
\begin{proof}
    
    From the definition of $F$, we have
    $$
    \begin{aligned}
    F(x^{k+1}) &= f(x^{k+1}) +g(x^{k+1}) \\
    &\leq f(x^k)+\left\langle\grad f(x^k), R^{-1}_{x^k}(x^{k+1})\right\rangle_{x^k} +\frac{L}{2}\left\|{R}_{x^k}^{-1}(x^{k+1})\right\|^2_{x^k}+g(x^{k+1}) \\
    &\leq f(x^k)+\left\langle\grad f(x^k), R^{-1}_{x^k}(x^{k+1})\right\rangle_{x^k} +\frac{L}{2}\left\|{R}_{x^k}^{-1}(x^{k+1})\right\|^2_{x^k}+(1-\lambda)g(x^k) +\lambda g(z^k) \\
    &= F(x^k) + \lambda\left\langle\grad f(x^k), R^{-1}_{x^k}(z^k)\right\rangle_{x^k} -\lambda g(x^k) +\lambda g(z^k) +\frac{L}{2}\left\|{R}_{x^k}^{-1}(x^{k+1})\right\|^2_{x^k}\\
    &= F(x^k) + \lambda \left(\left\langle\grad f(x^k), R^{-1}_{x^k}(z^k)\right\rangle_{x^k} - g(x^k) + g(z^k)\right) +\frac{L}{2}\left\|{R}_{x^k}^{-1}(x^{k+1})\right\|^2_{x^k}\\
    &= F(x^k)+ \lambda\theta (x^k) +\frac{L\lambda^2}{2}\left\|{R}_{x^k}^{-1}(z^k)\right\|^2_{x^k}\\
    &\leq F(x^k)+ \lambda\theta (x^k) +\frac{L\lambda^2}{2}\Omega^2,
    \end{aligned}
    $$
    where the first inequality comes from the $L$-smoothness of $f$ and the second inequality is due to the $R$-convexity of $g$. Then, from the definition of $\theta$ in \eqref{def theta}, and using the fact that $\left\|{R}_{x^k}^{-1}(x^{k+1})\right\|^2_{x^k} = \lambda^2 \left\|{R}_{x^k}^{-1}(z^k)\right\|^2_{x^k}$ and  Assumption~\ref{assu:omega}, we obtain the desired result.
\end{proof}

Additionally, we can derive the following proposition.
\begin{proposition}\label{prop:des}
    Let  $x^* \in \mathcal{X}$ be an optimal point of \eqref{eq:prob}, and assume that $f$ is $R$-convex. Then, Algorithm \ref{alg:RGCG} generates $\{x^k\}$ satisfying
    \begin{equation}
        F(x^{k+1})  \leq F(x^*)+\pi_k\left(1-\lambda_0\right)\left(F(x^0)-F(x^*)\right)+\sum_{s=0}^k \frac{\pi_k}{\pi_s} \frac{L\lambda_s^2}{2}\Omega^2,
    \end{equation}
    where $\pi_k:=\prod_{s=1}^k\left(1-\lambda_s\right)$ with $\pi_0=1$.
\end{proposition}
\begin{proof}
    From Lemma~\ref{lemma:des}, 
    we obtain
    $$F(x^{k+1}) \leq F(x^k)+ \lambda_k\theta (x^k) +\frac{L\lambda_k^2}{2}\Omega^2.$$
    Define $\Delta_k:=F(x^{k})-F(x^*)$. Then, we get
    $$
    \Delta_{k+1} \leq \Delta_k + \lambda_k \theta (x^k)+\frac{L\lambda_k^2}{2}\Omega^2.
    $$ 
    From the convexity of $f$, we obtain
    \begin{align*}
        F(x^*)-F(x^k) &\geq\left\langle\operatorname{grad} f(x^k), R_{x^k}^{-1}(x^*)\right\rangle_{x^k} + g(x^*) - g(x^k) \\&\geq \min_{u \in \mathcal{X}} \left \langle\grad f(x^k), R^{-1}_{x^k}(u)\right\rangle_{x^k}  +g(u) -g(x^k) = \theta (x^k).
    \end{align*}
    Then 
    we have $\Delta_k \leq -\theta (x^k)$, which gives
    \[
    \Delta_{k+1} \leq (1 - \lambda_k) \Delta_k + \frac{L \lambda_k^2}{2} \Omega^2.
    \]
    Now,  let us expand \(\Delta_k\) and gradually express \(\Delta_k\) in terms of \(\Delta_0\). Ultimately, for the recursive relation at step \(k+1\), we obtain:
    \[
    \Delta_{k+1} \leq \prod_{s=0}^k (1 - \lambda_s) \Delta_0 + \sum_{s=0}^k \frac{L \lambda_s^2}{2} \prod_{j=s+1}^k (1 - \lambda_j) \Omega^2,
    \]    
    that is,
    $$
    \Delta_{k+1} \leq \pi_k\left(1-\lambda_0\right) \Delta_0+\sum_{s=0}^k \frac{\pi_k}{\pi_s} \frac{L\lambda_s^2}{2}\Omega^2.
    $$
    
    \noindent Rearranging the terms completes the proof.
\end{proof}

Now we present the following lemma, which will be instrumental in our analysis of convergence and iteration complexity bounds for the CG method.

\begin{lemma} \label{lemma:ak} \cite[Lemma 13.13]{beck2017first}
    Let $\left\{a_k\right\}$ and $\left\{b_k\right\}$ be nonnegative sequences of real numbers satisfying
    $$
    a_{k+1} \leq a_k-b_k \beta_k+\frac{A}{2} \beta_k^2, \quad k=0,1,2, \ldots,
    $$
    where $\beta_k:=2 /(k+2)$ and $A$ is a positive number. Suppose that $a_k \leq b_k$ for all $k$. Then
    \begin{itemize}
        \item[(1)]  $a_k \leq \frac{2 A}{k}$ for all $k=1,2,\ldots$
        \item[(2)] $\min _{\ell \in\left\{\left\lfloor\frac{k}{2}\right\rfloor+2, \ldots, k\right\}} b_{\ell} \leq \frac{8 A}{k-2}$ for all $k=3,4, \ldots, \quad$ where $\lfloor k / 2\rfloor=\max \{n \in \mathbb{N}: n \leq k / 2\}$.
    \end{itemize} 
\end{lemma}

\subsection{Analysis with the adaptive and the diminishing step sizes}
Building upon the theoretical results established previously, we now present the main theorem regarding the convergence rate for two types of step sizes: the adaptive and the diminishing step sizes.
\begin{theorem}\label{thm:ada dimi}
    Let $x^*$ be an optimal point of~\eqref{eq:prob}, and assume that $f$ is $R$-convex. Then, the sequence of iterates $\{x^k\}$ generated by Algorithm \ref{alg:RGCG} with the adaptive or diminishing step size  satisfies $F(x^k)-F(x^*)=O(1 / k)$. Specifically, we obtain
    $$\min _{\ell \in\left\{\left\lfloor\frac{k}{2}\right\rfloor+2, \ldots, k\right\}} -\theta (x^k) \leq \frac{8 L\Omega^2}{k-2}.$$
\end{theorem}
\begin{proof}
    From Lemma \ref{lemma:des} and the diminishing step size $\lambda_k=\frac{2}{k+2} \in (0, 1) $, we have
    \begin{equation}\label{eq:Fdec}
        F(x^{k+1}) - F(x^{*}) \leq F(x^{k})-F(x^*)+ \lambda_k\theta (x^k) +\frac{L\lambda_k^2}{2}\Omega^2.
    \end{equation}
    
    \noindent Considering the adaptive step size
    $
    \eta_k
    =\operatorname{argmin}_{\lambda \in(0,1]}\left\{ \lambda\theta(x^k)+\frac{L}{2} \lambda^2 d^2(p(x^k), x^k) \right\}
    $, we have $\eta_k\theta (x^k) +\frac{L\eta_k^2}{2} d^2(p(x^k), x^k)\leq \lambda_k\theta (x^k) +\frac{L\lambda_k^2}{2}d^2(p(x^k), x^k)\leq \lambda_k\theta (x^k) +\frac{L\lambda_k^2}{2}\operatorname{diam}(\mathcal{X})^2\leq \lambda_k\theta (x^k) +\frac{L\lambda_k^2}{2}\Omega^2$, so the inequality \eqref{eq:Fdec} still holds. 

    Define $\Delta_k:=F(x^{k})-F(x^*)$. Then \eqref{eq:Fdec} is equivalent to
    $$
    \Delta_{k+1} \leq \Delta_k + \lambda_k \theta (x^k)+\frac{L\lambda_k^2}{2}\Omega^2.
    $$ 
    From the convexity of $f$, we obtain
    \begin{align*}
        F(x^*)-F(x^k) &\geq\left\langle\operatorname{grad} f(x^k), R_{x^k}^{-1}(x^*)\right\rangle_{x^k} + g(x^*) - g(x^k) \\&\geq\min_{u \in \mathcal{X}}  \left\langle\grad f(x^k), R^{-1}_{x^k}(u)\right\rangle_{x^k}  +g(u) -g(x^k) = \theta (x^k).
    \end{align*}
    Then we have $0\leq\Delta_k \leq -\theta (x^k)$. 
    So setting $a_k =\Delta_{k}$ and $b_k = -\theta (x^k) \geq 0$, $ \beta_k =\lambda_k$, and $A = L\Omega^2$ in Lemma~\ref{lemma:ak} we obtain
    $$
    F(x^{k})- F(x^*) \leq  \frac{2L}{k}\Omega^2,
    $$
    and
    $$\min _{\ell \in\left\{\left\lfloor\frac{k}{2}\right\rfloor+2, \ldots, k\right\}} -\theta (x^k) \leq \frac{8 L\Omega^2}{k-2},$$ for all $k=3,4, \ldots$, where $\lfloor k / 2\rfloor=\max \{n \in \mathbb{N}: n \leq k / 2\}$.
\end{proof}

The above theorem also indicates that the iteration complexity of the method is $\mathcal{O} \left( 1/k \right)$.
Note that a similar result can be obtained by following the proof of \cite[Corollary~5]{yu2017generalized}.

\subsection{Analysis with the Armijo step size}

The following lemma establishes the existence of intervals of step sizes that satisfy the Armijo condition.
\begin{lemma}\label{lemma:armijo}
    Let $\zeta \in(0,1)$, $x^k \in \mathcal{X}$ and $\theta(x^k) \neq 0$. Then, there exists $0<\bar{\lambda} \leq 1$ such that
    \begin{equation}\label{ineq armijo}
        F(R_{x^k} (\lambda R_{x^k}^{-1} (p(x^k)))) \leq F(x^k)-\zeta \lambda\left|\theta (x^k)\right|  \quad \forall \lambda \in(0, \bar{\lambda}] .
    \end{equation}
\end{lemma}
\begin{proof}
    Since $f$ is differentiable and $g$ is $R$-convex, the following holds for all $\lambda \in(0,1)$, 
    $$
    \begin{aligned}
    & F (R_{x^k} (\lambda R_{x^k}^{-1} (p(x^k)))) \\
    = & g (R_{x^k} (\lambda R_{x^k}^{-1} (p(x^k)))) + f(R_{x^k} (\lambda (R_{x^k}^{-1} (p(x^k))))) \\
    \leq & (1-\lambda) g (x^k) + \lambda g (p(x^k)) + f(x^k) + \lambda \left\langle \grad f (x^k), R_{x^k}^{-1} (p(x^k)) \right\rangle_{x^k} + {o}(\lambda) \\
    = & F (x^k) + \lambda \left( \left\langle \grad f (x^k), R_{x^k}^{-1} (p(x^k)) \right\rangle_{x^k} + g (p(x^k)) - g (x^k) \right) + {o}(\lambda) \\
    = & F(x^k) + \lambda \zeta \theta(x^k) + \lambda \left( (1-\zeta) \theta(x^k) + \frac{{o}(\lambda)}{\lambda} \right).
    \end{aligned}
    $$
    The inequality arises from the Taylor expansion of \(f\) and the $R$-convexity of \(g\).
    Since $\lim _{\lambda \rightarrow 0} \frac{{o}(\lambda)}{\lambda}=0$, $\zeta \in(0,1)$ and $\theta(x^k) < 0$ from the definition as in~\eqref{def theta}, there exists $\bar{\lambda} \in(0,1]$ such that $(1-\zeta) \theta(x^k)+\frac{{o}(\lambda)}{\lambda} \leq 0$,
    then \eqref{ineq armijo} holds for all $\lambda \in(0, \bar{\lambda}]$.
\end{proof}

In the subsequent results, we assume that the sequence generated by Algorithm~\ref{alg:RGCG} is infinite, which means that $\theta(x^k)<0$ for all $k$.

\begin{assumption}
\label{assum lowerbounded}
    Define \(\mathcal{L} = \{ x \mid F(x) \leq F(x^0) \}\), and assume that all monotonically nonincreasing sequences in \(F(\mathcal{L})\) are bounded from below.
\end{assumption}

    \begin{theorem}\label{thm armijo1}
        Assume that \(F\) satisfies Assumption~\ref{assum lowerbounded}
        and the sequence \(\{x^k\}\) is generated by Algorithm~\ref{alg:RGCG} with the Armijo step size. The following statements hold:
        \begin{itemize}
            \item[(1)] $\lim _{k \rightarrow \infty} \theta(x^k)=0$;
            \item[(2)]  Let $x^*$ be an accumulation point of the sequence $\{x^k\}$. Then $x^*$ is a stationary point of \eqref{eq:prob} and $\lim _{k \rightarrow \infty} F(x^k)=F(x^*)$.
        \end{itemize}
    \end{theorem}
        \begin{proof}
        (1) By the Armijo step size, the fact that $\theta(x^k)<0$ for all $k \in \mathbb{N}$, and by Lemma~\ref{lemma:armijo}, we have
        \begin{equation}\label{eq FFDes}
            F(x^k)- F(x^{k+1}) \geq \zeta \lambda_k\left|\theta(x^k)\right|>0,
        \end{equation}
        which implies that the sequence $\{F(x^k)\}$ is monotonically decreasing. On the other hand, since $\{x^k\} \subset \mathcal{X}$ and $\mathcal{X}$ is compact, there exists $\bar{x} \in \mathcal{X}$ that is a accumulation point of a subsequence of $\{x^k\}$. Without loss of generality, assume that $x^k \rightarrow \bar{x}$.

        From \eqref{eq FFDes}, with $k =0,1,\ldots, N$ for some $N$, we have $F(x^0)-F_{min} \geq F(x^0)-F(x^N) \geq \zeta  \sum_{k=0}^{N} \lambda_k\left|\theta (x^k)\right|$, where $F_{min}$ exists from Assumption~\ref{assum lowerbounded}. Taking $N\rightarrow \infty$, we have $\sum_{k=0}^{\infty} \lambda_k\left|\theta (x^k)\right| < +\infty$. 
        Therefore, we have $\lim _{k \rightarrow\infty}\lambda_k\theta (x^k)=0$.
        If
        \begin{equation}\label{theta 0}
            \lim _{k \rightarrow \infty}\theta (x^k)=0,
        \end{equation}
        then the assertion holds.
        So we consider the case 
        \begin{equation}\label{lambda 0}
            \lim _{k \rightarrow \infty}\lambda_k =0.
        \end{equation}
        Assume that $\lim _{ \mathbb{K}_1 \ni k \to \infty} \lambda_k=0$ for some $\mathbb{K}_1 \subseteq \mathbb{N}$. 
        Suppose by contradiction that $\theta(\bar{x})<0$. Then, there exist $\delta>0$ and $\mathbb{K}_2 \subseteq \mathbb{K}_1$ such that
        $$
        \theta (x^k)<-\delta \quad \forall k \in \mathbb{K}_2 .
        $$
        Without loss of generality, we assume the existence of \(\bar{p} \in \mathcal{X}\) such that:
        $$
        \lim _{\mathbb{K}_2 \ni k \to \infty} p (x^k)=\bar{p}.
        $$
        Since $\lambda_k<1$ for all $k \in \mathbb{K}_2$, by Armijo step size, there exists $\bar{\lambda}_k \in (0, \lambda_k / \omega_1)$ such that
        $$
        F (R_{x^k}  (\bar{\lambda}_k R_{x^k}^{-1}(p (x^k))))>F (x^k)+\zeta \bar{\lambda}_k \theta (x^k) \quad \forall k \in \mathbb{K}_2,
        $$
        Then, we have
        \begin{equation}\label{armijoK2}
            \frac{F (R_{x^k}  (\bar{\lambda}_k R_{x^k}^{-1}( p (x^k))))-F (x^k)}{\bar{\lambda}_k}>\zeta \theta (x^k) \quad \forall k \in \mathbb{K}_2 .
        \end{equation}
        On the other hand, since $0<\bar{\lambda}_k \leq 1$, and $g$ is $R$-convex, then
        $$
        g (R_{x^k}  (\bar{\lambda}_k R_{x^k}^{-1}(p (x^k))) )\leq  (1-\bar{\lambda}_k)g (x^k) + \bar{\lambda}_k g (p (x^k)),
        $$ 
        namely,
        \begin{equation}\label{eq: gconvex}
        \frac{g (R_{x^k}  (\bar{\lambda}_k R_{x^k}^{-1}(p (x^k)))-g (x^k)}{\bar{\lambda}_k} \leq g (p (x^k))-g (x^k) \quad \forall k \in \mathbb{K}_2.
        \end{equation}
        Since $f$ is differentiable and $\lim _{\mathbb{K}_2 \ni k \to \infty} \bar{\lambda}_k=0$, for all $k \in \mathbb{K}_2$, we have
        \begin{equation}\label{eq: simu f}
            f (R_{x^k}  (\bar{\lambda}_k R_{x^k}^{-1}(p (x^k)))) = f (x^k)+\bar{\lambda}_k\left\langle\operatorname{grad} f (x^k), R_{x^k}^{-1}(p (x^k))\right\rangle_{x^k} + o (\bar{\lambda}_k),
        \end{equation}
        which, together with \eqref{eq: gconvex} and
        \eqref{eq: simu f} and the definition of $\theta$, gives
        \begin{equation}\label{eq theta f}
        \begin{aligned}
        \theta (x^k) & = \left\langle\operatorname{grad} f (x^k), R_{x^k}^{-1}(p (x^k))\right\rangle_{x^k} +g (p (x^k))-g (x^k)\\
        & \geq \frac{f (R_{x^k}  (\bar{\lambda}_k R_{x^k}^{-1}(p (x^k)))) -f (x^k)- o (\bar{\lambda}_k)}{\bar{\lambda}_k} +\frac{g (R_{x^k}  (\bar{\lambda}_k R_{x^k}^{-1}(p (x^k)))-g (x^k)}{\bar{\lambda}_k}\\
        & \geq \frac{F (R_{x^k}  (\bar{\lambda}_kR_{x^k}^{-1}(p (x^k))))-F (x^k)}{\bar{\lambda}_k}-\frac{o (\bar{\lambda}_k)}{\bar{\lambda}_k} \quad \forall k \in \mathbb{K}_2 .
        \end{aligned}
        \end{equation}
        Combining \eqref{eq theta f} with \eqref{armijoK2}, we obtain for all  $k \in \mathbb{K}_2$,
        $$
        \frac{F (R_{x^k}  (\bar{\lambda}_kR_{x^k}^{-1}(p (x^k))))-F (x^k)}{\bar{\lambda}_k\zeta}
            > \frac{F (R_{x^k}  (\bar{\lambda}_kR_{x^k}^{-1}(p (x^k))))-F (x^k)}{\bar{\lambda}_k}-\frac{o (\bar{\lambda}_k)}{\bar{\lambda}_k},
            $$
        that is,
            $$
            \left(\frac{1}{\zeta} -1\right)\frac{F (R_{x^k}  (\bar{\lambda}_kR_{x^k}^{-1}(p (x^k))))-F (x^k)}{\bar{\lambda}_k}
            > -\frac{o (\bar{\lambda}_k)}{\bar{\lambda}_k}.
        $$
        Then, we have
        $$
        \frac{F(R_{x^k} \bar{\lambda}_k (R_{x^k}^{-1} (p(x^k))))-F(x^k)}{\bar{\lambda}_k}>\left(\frac{-\zeta}{1-\zeta}\right) \frac{o (\bar{\lambda}_k)}{\bar{\lambda}_k} \quad \forall k \in \mathbb{K}_2 .
        $$
        On the other hand, from the fact that $\theta (x^k)<-\delta$ and \eqref{eq theta f}, we obtain
        $$
        -\delta+\frac{o (\bar{\lambda}_k )}{\bar{\lambda}_k}>\frac{F(R_{x^k} \bar{\lambda}_k R_{x^k}^{-1} (p(x^k)))-F(x^k)}{\bar{\lambda}_k} \quad \forall k \in \mathbb{K}_2 .
        $$
        Combining the two results above, it holds that 
        $$
        -\delta+\frac{o(\bar{\lambda}_k)}{\bar{\lambda}_k}>\left(\frac{-\zeta}{1-\zeta}\right) \frac{o(\bar{\lambda}_k)}{\bar{\lambda}_k} \quad \forall k \in \mathbb{K}_2 .
        $$
        
        \noindent Since $\frac{o(\bar{\lambda}_k)}{\bar{\lambda}_k} \to 0$ as $k \to \infty$, we obtain $\delta \leq0$, which contradicts the fact that $\delta>0$. Thus, $\theta\left(\bar{x}\right)=0$.

        (2) 
        Assume that \( x^* \),  an accumulation point of the sequence \(\{x^k\}\), is not a stationary point of $F$, namely \(\theta\left(x^*\right)<0\). By the definition of \(\theta\) in \eqref{def theta}, we have
        \begin{align*}
            \theta\left(x^k\right) &= \left\langle \grad f(x^k), R^{-1}_{x^k}(p(x^k)) \right\rangle_{x^k} + g(p(x^k)) - g(x^k)\\
            &\leq \left\langle \grad f(x^k), R^{-1}_{x^k}(u) \right\rangle_{x^k} + g(u) - g(x^k) \quad \forall u \in \mathcal{X},
        \end{align*}
        and by taking the limit we obtain
        $$
        \limsup_{k \rightarrow \infty} \theta\left(x^k\right) \leq \left\langle \grad f(x^*), R^{-1}_{x^*}(u) \right\rangle_{x^*} + g(u) - g(x^*) \quad \forall u \in \mathcal{X},
        $$
        since $f$ is continuously differentiable, $g$ is lower semicontinuous and Assumption~\ref{assum1} holds.
        Now taking the minimum over $x\in \mathcal{X}$ and because \( x^* \) is not stationary point of $F$, 
        $$
        \limsup_{k \rightarrow \infty} \theta\left(x^k\right) \leq \min_{u \in \mathcal{X}} \left\langle \grad f(x^*), R^{-1}_{x^*}(u) \right\rangle_{x^*} + g(u) - g(x^*)=\theta\left(x^*\right)<0,
        $$
        which contradicts to $\lim _{k \rightarrow \infty} \theta(x^k)=0$, 
        so \( x^* \) is a stationary point of $F$.
        
        Furthermore, from the monotonicity of the sequence $\{F(x^k)\}$ given in \eqref{eq FFDes}, we have $\lim _{k \rightarrow \infty} F(x^k)=F\left(x^*\right)$, which completes the proof.
    \end{proof}

The following theorem provides the iteration complexity of Algorithm~\ref{alg:RGCG} with Armijo step size, indicating that after a finite number of steps, the iteration complexity of the method could be $\mathcal{O} \left( 1/\epsilon^2 \right)$.

    \begin{theorem}\label{thm armijo2}

        Assume that  $F$  satisfies Assumptions \ref{assu Lsmooth} and ~\ref{assum lowerbounded}. Then, there exists a finite integer $J\in \mathbb{N}$ such that Algorithm~\ref{alg:RGCG} with Armijo step size achieves an $\epsilon$-optimal solution within  $J + \mathcal{O} \left( 1/\epsilon^2 \right)$  iterations.

    \end{theorem}

    \begin{proof}
        First, we consider that $0<\lambda_k<1$. From the Armijo step size condition, there exists \(0 < \bar{\lambda}_k \leq \min \left\{1, \frac{\lambda_k}{\omega_1}\right\}\) such that
        $$
        F(R_{x^k} \bar{\lambda}_k(R_{x^k}^{-1}(p (x^k))))>F(x^k)+\zeta \bar{\lambda}_k \theta(x^k) .
        $$
        
        Recall that from Lemma \ref{lemma:des}, we have
        $$
        F(R_{x^k} (\bar{\lambda}_k(R_{x^k}^{-1}(p (x^k))))) < F(x^k)+ \bar{\lambda}_k\theta (x^k) +\frac{L\bar{\lambda}_k^2}{2}\Omega^2.
        $$
        Combining the above two inequalities, we obtain
        $$
        \zeta \bar{\lambda}_k\theta(x^k)< \bar{\lambda}_k\theta (x^k) +\frac{L\bar{\lambda}_k^2}{2}\Omega^2,
        $$
        namely,
        $$
        0 < -\frac{2(1-\zeta)}{L\Omega^2}\theta (x^k) < \bar{\lambda}_k\leq \frac{\lambda_k}{\omega_1}.
        $$
        Defining $\gamma=\frac{2\omega_1(1-\zeta)}{L\Omega^2}$, we get
        \begin{equation}\label{ineq lambda}
            0 < -\gamma\theta (x^k) < \lambda_k.
        \end{equation}

        From Theorem~\ref{thm armijo1}, we have $\lim _{k \rightarrow \infty} \theta(x^k)=0$. Then, without loss of generality, for sufficiently large \(J\), when \(k > J\), \(\left|\theta(x^k)\right| \leq 1/\gamma \). Then \eqref{ineq lambda} still holds for $\lambda_k =1$. 
        Taking it into \eqref{eq FFDes}, we conclude that
        \begin{equation}\label{eq gamaff}
            0 < \zeta \gamma\left|\theta(x^k)\right|^2  \leq F(x^k)-F(x^{k+1}).
        \end{equation}
        Namely \eqref{eq gamaff} holds for all $k > J$.
        By summing both sides of the second inequality in \eqref{eq gamaff} for $k= J+1, \ldots, N$, we have
        $$
        \sum_{k=J+1}^{N}\left|\theta(x^k)\right|^2 \leq \frac{1}{\zeta \gamma}(F(x^{J+1})-F(x^{N+1}))\leq \frac{1}{\zeta \gamma}(F(x^0)-F(x^{*})),
        $$
        where $x^*$ is an accumulation point of $\{x^k\}$, as in Theorem~\ref{thm armijo1}.
        
        This implies $\min_k \left\{\left|\theta(x^k)\right| \mid k= J+1, \ldots, N\right\} \leq \sqrt{(F(x^0)-F(x^{*})) /(\zeta \gamma (N-J))}$.
        Namely, Algorithm \ref{alg:RGCG} generates an iterate $x^k$ satisfying $\left|\theta(x^k)\right| \leq \epsilon$ in at most $J + (F(x^0)-F\left(x^*\right)) /(\zeta \gamma \epsilon^2)$ iterations.
    \end{proof}

    Furthermore, we can obtain stronger results by making additional assumptions about the objective function $g$, similarly to \cite[Section 3 (A4)]{assunccao2023generalized}.
    \begin{assumption}\label{assu:Lg}
    \begin{sloppypar}
    Assume the function $g$ is Lipschitz continuous with constant $L_g>0$ in $\mathcal{X}$, namely, $|g(x) - g(y)| \leq L_g \left\|{R}_x^{-1}(y)\right\|_x$. 
    \end{sloppypar}
    \end{assumption}

    Moreover, we define
    \begin{equation}\label{rho gamma}
    \begin{aligned}
    & \rho:= \sup \left\{\left\|\operatorname{grad} f(x)\right\|_x \mid x \in \mathcal{X} \right\} \\
    & \gamma:=\min \left\{\frac{1}{(L_g+\rho) \Omega}, \frac{2 \omega_1(1-\zeta)}{L \Omega^2}\right\} .
    \end{aligned}
    \end{equation}
Next, we introduce the following lemma, which is crucial for the subsequent convergence rate analysis.
    \begin{lemma}\label{lemma:lambda}
        Let $\{x^k\}$ be a sequence generated by Algorithm~\ref{alg:RGCG} with Armijo step size. Assume that  Assumptions \ref{assu Lsmooth}, \ref{assu:omega} and \ref{assu:Lg} hold. Then $\lambda_k\geq \gamma \left|\theta\left(x^k\right)\right| > 0.$
    \end{lemma}
    \begin{proof}
        Since \(\lambda_k \in (0,1]\), we consider two cases: \(\lambda_k = 1\) and \(0 < \lambda_k < 1\).  
        First, assume that \(\lambda_k = 1\).
        From \eqref{def theta}, we have
        $$
        0<-\theta (x^k) = g (x^k)-g (p (x^k))-\left\langle\operatorname{grad} f (x^k), R_{x^k}^{-1} (p (x^k))\right\rangle_{x^k},
        $$
        thus, it follows from Assumption~\ref{assu:Lg} and the Cauchy-Schwarz inequality that
        $$
        0<-\theta (x^k) \leq L_g \left\|{R}^{-1}_{x^k}(p (x^k))\right\|_{x^k} 
        +\left\|\operatorname{grad} f (x^k)\right\|_{x^k}\left\|R_{x^k}^{-1} (p(x^k))\right\|_{x^k} .
        $$
        
        Using \eqref{rho gamma}, we have $0<-\theta (x^k) \leq (L_g+\rho) \Omega$. Furthermore,
        $$
        0<\gamma |\theta (x^k)| = -\gamma \theta (x^k) \leq \frac{-\theta (x^k)}{ (L_g+\rho) \Omega} \leq 1,
        $$
        which demonstrates that the desired inequality holds.
        
        Now, consider $0<\lambda_k<1$. From the Armijo step size,  there exists $0<\bar{\lambda}_k \leq \min \left\{1, \frac{\lambda_k}{\omega_1}\right\}$ such that
        $$
        F (R_{x^k} (\bar{\lambda}_k R_{x^k}^{-1} (p (x^k))))>F (x^k)+\zeta \bar{\lambda}_k \theta (x^k) .
        $$
        By using Lemma \ref{lemma:des}, we have
        $$
        F (R_{x^k} (\bar{\lambda}_k R_{x^k}^{-1} (p (x^k)))) \leq F (x^k)+\bar{\lambda}_k \theta (x^k)+\frac{L}{2}\Omega^2\bar{\lambda}_k^2.
        $$
        From the above two inequalities, we have that
        $$
        -\theta (x^k)(1-\zeta)
        <\frac{L}{2}\Omega^2 \bar{\lambda}_k 
        \leq \frac{L}{2}\Omega^2 \frac{\lambda_k}{\omega_1} .
        $$
        Therefore, from \eqref{rho gamma}, we get
        $$
        0<\gamma |\theta (x^k)| = -\gamma \theta (x^k) \leq-\frac{2 \omega_1(1-\zeta)}{L \Omega^2} \theta (x^k)<\lambda_k,
        $$
        which concludes the proof.    
    \end{proof}

        \begin{theorem}
        Let $\{x^k\}$ be a sequence generated by Algorithm~\ref{alg:RGCG} with Armijo step size. Assume that Assumptions \ref{assu Lsmooth}, \ref{assu:omega} and \ref{assu:Lg} hold. Then $\lim _{k \rightarrow \infty} F(x^k)=F(x^*)$, for some $x^* \in$ $\mathcal{X}$. Moreover, the following statements hold:
        \begin{itemize}
            \item[(1)] $\lim _{k \rightarrow \infty} \theta (x^k )=0$;
            \item[(2)] $\min \left\{\left|\theta (x^k )\right| \mid k=0,1, \ldots, N-1\right\} \leq \sqrt{F(x^0)-F(x^{*}) /(\zeta \gamma N)}$.
        \end{itemize}    
    \end{theorem}
    \begin{proof}
        \begin{equation}\label{eq FDes}
            F(x^k)- F(x^{k+1}) \geq \zeta \lambda_k\left|\theta (x^k )\right|>0.
        \end{equation}
        It implies that the sequence $\{F (x^k)\}_{k \in \mathbb{N}}$ is monotone decreasing.
        Combined with Lemma \ref{lemma:lambda}, we have
        \begin{equation}\label{eq addd}
            0 < \zeta \gamma\left|\theta (x^k )\right|^2  \leq F (x^k )-F (x^{k+1} ).
        \end{equation}
         Since \( \{x^k\}_{k \in \mathbb{N}} \subset \mathcal{X} \) and \( \mathcal{X} \) is compact, there exists a limit point \( x^* \in \mathcal{X} \) of the sequence \( \{x^k\}_{k \in \mathbb{N}} \). Let \( \{x^{k_j}\}_{j \in \mathbb{N}} \) denote a subsequence of \( \{x^k\}_{k \in \mathbb{N}} \) such that \( \lim_{j \to \infty} x^{k_j} = x^* \).
         Since $ \{x^{k_j}\}_{j \in \mathbb{N}} \subset \mathcal{X}$ and $g$ is Lipschitz continuous, we have
        $$
        \left\|F (x^{k_j})-F (x^*)\right\|=\left\|g (x^{k_j})-g (x^*)+f (x^{k_j})-f (x^*)\right\| \leq L_g\left\|x^{k_j}-x^*\right\|+\left\|f (x^{k_j})-f (x^*)\right\|,
        $$
        for all $j \in \mathbb{N}$. Considering that $f$ is continuous and $\lim _{j \rightarrow \infty} x^{k_j}=x^*$, we conclude from the last inequality that $\lim _{j \rightarrow \infty} F (x^{k_j})=F (x^*)$. 
        Thus, from the monotonicity of the sequence $ \{F (x^k)\}_{k \in \mathbb{N}}$, we have $\lim _{k \rightarrow \infty} F (x^k)=F (x^*)$. Therefore, from \eqref{eq addd}, $\lim _{k \rightarrow \infty}\left|\theta (x^k)\right|=0$, which implies item (1). 

        By summing both sides of the second inequality in \eqref{eq addd} for $k=0,1, \ldots, N-1$, we conclude that
        $$
        \sum_{k=0}^{N-1}\left|\theta (x^k)\right|^2 \leq \frac{1}{\zeta \gamma}(F(x^0)-F(x^{N}))\leq \frac{1}{\zeta \gamma}(F(x^0)-F(x^{*})),
        $$
        which implies the item (2).
        
        Therefore, $x^k$ satisfies $\left|\theta (x^k)\right| \leq \epsilon$ in at most $ (F (x^0)-F (x^*)) / (\zeta \gamma \epsilon^2)$ iterations.
    \end{proof}

\begin{remark}
     It is worth noting that our framework naturally recovers the classical setting of geodesic convexity. Specifically, when the retraction in this paper is chosen to be the exponential map, our $R$-convexity requirement on the function $g$ reduces to standard geodesic convexity. 
     It follows that theoretical results (Lemma~\ref{lemma:des}, Proposition~\ref{prop:des}, and Theorem~\ref{thm:ada dimi}) remain valid without Assumption~\ref{assu:omega}, with the constant $\Omega$ in the convergence bounds being replaced by  $\operatorname{diam}(\mathcal{X})$.
    \end{remark}

\section{Solving the subproblem}\label{sec: subproblem}
In this section, we will propose an algorithm to solve the subproblem of the RGCG method.
Recall that in Algorithm~\ref{alg:RGCG}, given the current iterate $x$,
we need to compute $p(x)$ and $\theta(x)$ as
\begin{equation}\label{pro pxk}
    p(x) = \arg\min_{u \in \mathcal{X}}  \left\langle\grad f(x), R^{-1}(u)\right\rangle_{x}  +g(u) -g(x),
\end{equation}
\begin{equation}
    \theta(x) =  \left\langle\grad f(x), R^{-1}_{x}(p(x)) \right\rangle_{x} +g(p(x)) -g(x).
\end{equation}
Now define $\tilde{p}(x)$ as the descent direction on the tangent space, i.e.,
\begin{equation}\label{pro dk}
    \tilde{p}(x) = \arg\min_{d \in T_{x}\mathcal{M}}  \left\langle\grad f(x), d \right\rangle_{x}  +g(R_{x}(d)).
\end{equation}
We also define
\begin{equation}
    \ell_x(d) =  \left\langle\grad f(x), d \right\rangle_{x} +g(R_{x}(d)). 
\end{equation}
We observe that for both \eqref{pro pxk} and \eqref{pro dk}, the retraction is not necessarily a convex function, which implies that the resulting optimization problem is not necessarily convex.
Non-convex problems are inherently more challenging to solve. To address this, we aim to transform the subproblem into a series of convex problems, similar to the general method for solving Riemannian proximal mappings proposed in \cite{huang_riemannian_2022}. To ensure that we can locally approximate and simulate the subproblems effectively, we consider the following assumption:

\begin{assumption}\label{assump6}
 Assume that the manifold \(\mathcal{M}\) is either an embedded submanifold of \(\mathbb{R}^n\) or a quotient manifold whose total space is an embedded submanifold of \(\mathbb{R}^n\).
\end{assumption} 

Let us now present a way to find the approximation of $\tilde{p}(x) = \arg\min_{d \in T_{x}\mathcal{M}} \ell_x(d)$. Assume that $d^0  \in T_{x}\mathcal{M}$, and let \( d^k \) denote the current estimate of \( \tilde{p}(x) \). Observe that
\begin{align*}
    \ell_x(d^k+\tilde{\xi}_k)
    &= \left\langle\operatorname{grad} f\left(x\right), d^k +\tilde{\xi}_k\right\rangle_{x} + g(R_{x}(d^k+\tilde{\xi}_k)) \\
    &=\left\langle\operatorname{grad} f\left(x\right), d^k \right\rangle_{x}+\left\langle\operatorname{grad} f\left(x\right), \tilde{\xi}_k\right\rangle_{x} + g(R_{x}(d^k+\tilde{\xi}_k)),
\end{align*}
for any $\tilde{\xi}_k \in \mathrm{T}_x \mathcal{M}$.
Let $y_k=R_x\left(d^k\right)$ and $\xi_k=\mathcal{T}_{R_{d^k}} \tilde{\xi}_k$. By the smoothness property of the retraction $R$, 
we obtain $R_x (d^k+\tilde{\xi}_k)=y_k+\xi_k+O\left(\left\|\xi_k\right\|_x^2\right)$, where $y=x+O(z)$ means $\lim \sup _{z \rightarrow 0}\|y-x\| /\|z\|<\infty$. Then, we obtain:
$$
\begin{aligned}
\ell_x(d^k+\tilde{\xi}_k)
= & \left\langle\operatorname{grad} f(x), d^k\right\rangle_x +\left\langle\operatorname{grad} f(x), \mathcal{T}_{R_{d^k}}^{-1} \xi_k\right\rangle_x+g\left(y_k+\xi_k+O\left(\left\|\xi_k\right\|_x^2\right)\right)  \\
= & \left\langle\operatorname{grad} f(x), d^k\right\rangle_x  +\left\langle\operatorname{grad} f(x), \mathcal{T}_{R_{d^k}}^{-1} \xi_k\right\rangle_x+g\left(y_k+\xi_k\right)+O\left(\left\|\xi_k\right\|_x^2\right),  
\end{aligned}
$$
where the second equality follows from Assumption~\ref{assu:Lg}, which states the Lipschitz continuity of \(g\). Note that the expression
\begin{equation}\label{sub_tran}
\left\langle\operatorname{grad} f(x), \mathcal{T}_{R_{d^k}}^{-1} \xi_k\right\rangle_{x}+g\left(y_k+\xi_k\right)
\end{equation}
can be viewed as a simple local model of \( \ell_x\left(d^k + \mathcal{T}_{R_{d^k}}^{-1} \xi_k\right) \). 
Thus, to obtain a new estimate from \( d^k \), we first determine a search direction \( \xi_k^* \) by minimizing \eqref{sub_tran} on \( \mathrm{T}_{y_k} \mathcal{M} \). Then, we update \( d^k \) in the direction of \( \mathcal{T}_{R_{d^k}}^{-1}(\xi_k^*) \).
As described in Algorithm~\ref{alg: subproblem}.

It is easy to find that Algorithm~\ref{alg: subproblem} is the application of the Riemannian GCG method for problem \eqref{pro dk} with a proper retraction. Moreover, note that \eqref{eq subsub} can be solved by the subgradient method or proximal gradient method.

 \begin{algorithm}[H]
      \caption{Solving the subproblem}
      \label{alg: subproblem}
      \noindent Require: An initial iterate $d^0  \in T_{x}\mathcal{M}$; a small positive constant $\sigma$; \\
        1: for $k=0, 1, \ldots$,  do\\
        2: \quad $y_k=R_x\left(d^k\right)$\\
        3: \quad Compute $\xi_k^*$ by solving 
        \begin{equation}\label{eq subsub}
            \xi_k^*=\arg \min _{\xi \in \mathrm{T}_{y_k} \mathcal{M}}  \left\langle\mathcal{T}_{R_{d^k}}^{-\sharp}\left(\operatorname{grad} f(x)\right), \xi\right\rangle_{y_k}+g\left(y_k+\xi\right);
        \end{equation}
        4: \quad$\alpha=1$\\
        5: \quad while \quad$\ell_x\left(d^k+\alpha \mathcal{T}_{R_{d^k}}^{-1} \xi_k^*\right) \geq \ell_x\left(d^k\right)-\sigma \alpha\left\|\xi_k^*\right\|_x^2$ \quad do\\
        6: \quad \quad $\alpha=\frac{1}{2} \alpha $; \\
        7: \quad end while\\
        8: \quad $d_{k+1}=d^k+\alpha \mathcal{T}_{R_{d^k}}^{-1} \xi_k^*$;\\
        4: end for
\end{algorithm}

\section{Accelerated conditional gradient methods}\label{sec: accelerate}
Now we consider the accelerated version of the RCG method. We noticed that based on the accelerated gradient methods, Li et al.  \cite{li2021momentum} establish links between subproblems of conditional gradient methods and the notion of momentum. They proposed a momentum variant of CG method and proved that it converges with a fast rate of $O(\frac{1}{k^2})$. However, the extension of this work to the Riemannian manifold case is not trivial, even though there are works on Nesterov accelerated methods on Riemannian manifolds \cite{alimisis2021momentum}.  Based on these existing works,  here we propose an accelerated CG method on Riemannian manifolds for composite function.

\begin{algorithm}[H]
    \caption{Riemannian accelerated GCG method }\label{alg:RMGCG3}
    \textbf{Step 0. Initialization}
     Choose $x^0 = p_0\in \mathcal{X}$, $\tau^k = \frac{2}{k+3}$ and initialize $k = 0$, $\beta_0=\mathbf{0}, d^0 =\mathbf{0}$.

     \textbf{Step 1. }
     Compute
     $$y_k = R_{x^k}(\tau^k R^{-1}_{x^k}(d^k)) $$
     $$\beta_{k+1} = \mathcal{T}_{y_k}(1-\tau^k)\beta_k + \lambda_k\grad f(y_k)$$
     
     \textbf{Step 2. }
     Compute $p_k$ and $\theta(x^k)$ as
     \[ p_{k+1} = \arg\min_{u \in \mathcal{X}}  \left\langle\beta_{k+1}, R_{y_k}^{-1}(u)\right\rangle_{y^k}  +g(u) -g(y_k), \] 
    $$d_{k+1}= R_{y_k}^{-1}(p_{k+1}).$$
    
     \textbf{Step 3. Stopping criteria}
     If $\theta(x^k) = 0$, then stop.
    
     \textbf{Step 4. Compute the iterate}
     Set
     \[ x^{k+1} = R_{x^k} ( \lambda_k R^{-1}_{x^k}(p_{k+1})). \] 
    
     \textbf{Step 5. Starting a new iteration}
     Set $k = k + 1$ and return to Step 1.
\end{algorithm}

Analyzing the convergence and convergence rate for the above accelerated method is challenging, primarily due to the inherent nonconvexity of the problem and
the complexities introduced by transporting tangent vectors.
A comprehensive analysis of the accelerated method is beyond the scope of this paper.
However, the performance of the accelerated method compared to the non-accelerated method is shown in the numerical experiments in Section~\ref{sec: expreriment}.

\section{Numerical experiments}
\label{sec: expreriment}

Sparse principal component analysis (Sparse PCA) plays a crucial role in high-dimensional data analysis by enhancing the interpretability and efficiency of dimensionality reduction. Unlike traditional PCA, which often produces complex results in high-dimensional settings, Sparse PCA introduces sparsity to refine data representation and improve clarity. This is particularly relevant in fields such as bioinformatics, image processing, and signal analysis, where interpretability and reconstruction accuracy are essential. In this section, we consider Sparce PCA problems to check the validity of our proposed methods. 
We explore two models for sparse PCA as in~\cite{huang_riemannian_2022}; see that reference for further analysis. Here, we set $\mathcal{X} = \mathcal{M}$ in each problem, and Assumptions~\ref{assu Lsmooth}, \ref{assum lowerbounded} and~\ref{assu:Lg} hold.

\subsection{Sparse PCA on the sphere manifold}
Now consider problems in sphere manifolds $\mathbb{S}^{n-1}=\{x \in \mathbb{R}^n:\|x\|=1\}.$ 
The specific optimization problem is formulated as follows:
\begin{equation}\label{exam: sphere}
    \min _{x \in \mathbb{S}^{n-1}}  - x^T A^T A x + \lambda \|x\|_1, 
\end{equation}
where $\norm{\,\cdot\,}_1$ denotes the $L_1$-norm, and $\lambda > 0$ is a parameter.
Here, we use the exponential mapping as the retraction, i.e.,
$$
\operatorname{Exp}_x\left(\eta_x\right)=x \cos \left(\left\|\eta_x\right\|\right)+\eta_x \sin \left(\left\|\eta_x\right\|\right) /\left\|\eta_x\right\|, \quad x \in \mathbb{S}^{n-1}, \quad \eta_x \in \mathrm{T}_x \mathbb{S}^{n-1}.
$$
The inverse exponential mapping can be defined as follows, as presented in~\cite{huang_riemannian_2022}:
$$
\log _x(y)=\operatorname{Exp}_x^{-1}(y)=\frac{\cos ^{-1}\left(x^T y\right)}{\sqrt{1-\left(x^T y\right)^2}}\left(I_n-x x^T\right) y, \quad x, y \in \mathbb{S}^{n-1}.
$$

The subproblem \eqref{eq pk} applied to \eqref{exam: sphere}, can be written as
\begin{equation}
    \min _{y \in \mathbb{S}^{n-1}} u(y), \quad \text {where } u(y)=\frac{1}{\lambda}\langle\grad f(x), \log _x(y)\rangle +\|y\|_1 .
\end{equation}
Let us not check how to compute its solution. 
Let \( h(y) = \frac{1}{\lambda} \langle \operatorname{grad} f(x), \log_x(y) \rangle \), and let \( y^k \) denote the current estimate of the minimizer of \( u(y) \). Then solving the following problem to obtain the next estimate \( y^{k+1} \):
$$
\min _{y \in \mathbb{S}^{n-1}} h(y^k)+\nabla h(y^k)^T (y-y^k)+\|y\|_1,
$$
which is equal to 
$$
\min _{y \in \mathbb{S}^{n-1}} \nabla h(y^k)^T y+\|y\|_1.$$

Actually, \( h(y) \) is approximated by its first-order Taylor expansion around \( y^k \). Note that since \( \|y\| = 1 \) for all \( y \in \mathbb{S}^{n-1} \), the above problem is further equivalent to:
\begin{equation}\label{spheresubproblen}
    \min _{y \in \mathbb{S}^{n-1}} \frac{1}{2}\left\|y+\nabla h(y^k)\right\|^2+\|y\|_1.
\end{equation}
The above problem can be resolved directly using \cite[Lemma D.1]{huang_riemannian_2022}.

\subsection{Sparse PCA on the Stiefel manifold}
Now consider sparse PCA on the Stiefel manifold. The specific optimization problem is formulated as follows:

\begin{equation}\label{ex:st}
    \min _{X \in \operatorname{st}(p, n)}-\operatorname{trace}\left(X^T A^T A X\right)+ \lambda \|X\|_1.
\end{equation}
where $ \operatorname{st}(p, n)=\left\{X \in \mathbb{R}^{n \times p} \mid X^T X=I_p\right\}$ denotes the Stiefel manifold, and $\lambda > 0$.
Moreover, the term $-\operatorname{trace}\left (X^T A^T A X\right)$ aims to maximize the variance of the data, \(\|X\|_1\) characterizes the sparsity of the principal components, and \( A \) is the original data matrix.
The tangent space  is defined as
$
\mathrm{T}_X \operatorname{st}(p, n)=\left\{\eta \in \mathbb{R}^{n \times p} \mid X^T \eta+\eta^T X=0\right\} .
$

Here we use the Riemannian metric defined as
$
\left\langle\eta_X, \xi_X\right\rangle_X=\operatorname{trace}(\eta_X^T \xi_X).
$
And using the polar decomposition as the retraction, i.e.,
$$
R_X\left(\eta_X\right)=\left(X+\eta_X\right)\left(I_p+\eta_X^T \eta_X\right)^{-1 / 2}.
$$
The vector transport is defined by
$$
\mathcal{T}_{\eta_X} \xi_X=Y \Omega+\left(I_n-Y Y^T\right) \xi_X\left(Y^T\left(X+\eta_X\right)\right)^{-1},
$$
where $Y=R_X\left(\eta_X\right)$ and $\Omega$ is the solution of the Sylvester equation 
$\left(Y^T\left(X+\eta_X\right)\right) \Omega+\Omega\left(Y^T(X+\eta_X)\right)=Y^T \xi_X-\xi_X^T Y
$, as given in \cite[Lemma 10.2.1]{huang2013optimization}.

For ease of notation, we define the operator $\mathcal{A}_k$ by $\mathcal{A}_k(V):=V^T X_k+X_k^T V$ and rewrite the subproblem \eqref{eq subsub} as
$$
\begin{aligned}
V_k:=\underset{V}{\operatorname{argmin}} & \left\langle \mathcal{T}_{R_{d^k}}^{-\sharp}\grad f\left(X_k\right), V\right\rangle+g\left(X_k+V\right) \\
\text { s.t. } & \mathcal{A}_k(V)=0 .
\end{aligned}
$$
This can be solved by the subgradient method or the proximal gradient method.
For the subgradient method, note that the subgradient of \( g(X) \) can be calculate through function $m(X)$, which is defined as:
\begin{itemize}[-]
\item \( m_{ij}(X) = \lambda \) if \( (X)_{ij} > 0 \),
\item \( m_{ij}(X) = -\lambda \) if \( (X)_{ij} < 0 \),
\item \( m_{ij}(X) \in [-\lambda, \lambda] \) if \( (X)_{ij} = 0 \),
\end{itemize}
where $m_{ij}(X)$ is the element of $m(X)$.
Then the subgradient of the above problem can be written as
$
\mathcal{T}_{R_{d^k}}^{-\sharp} \grad f(X_k) + \lambda \cdot \operatorname{sign}(X_k + V_k).
$

\subsection{Numerical results}
We implemented the algorithms in Python 3.9.12, and the experiments were carried out on a MacBook Pro with an Apple M1 Pro chip and 16GB of RAM.
\subsubsection{Sphere experiment}
 For the example \eqref{exam: sphere}, we conducted experiments on a $n$-dimensional problem, using a randomly generated \( n \times n \) standardized \( A \) and a regularization parameter \( \lambda = 0.1 \). The experiment evaluates Armijo step size with parameters \( \zeta = 0.1 \), \( \omega_1 = 0.05 \), and \( \omega_2 = 0.95 \). Each experiment runs for the same original problem 10 times with different initial points and stops if the norm of \(\theta_{x^k}\) is less than \(10^{-4}\) or if the difference between the latest function value and the function value from five iterations ago is less than \(10^{-4}\). 
 In this example, our experiments were conducted on high-dimensional spheres, with dimensions set to 10, 100, and 1000, across ten runs for each method. In the numerical experiments, we observed that solving the subproblems typically converges after just one iteration, suggesting that the local model provides an effective approximation of the original subproblem.

As shown in Table \ref{table 1} and Figure~\ref{fig:example_sphere}, the results demonstrated that the Armijo step-size strategy consistently outperformed the others in terms of computational time and required iterations to achieve function value convergence. 
Overall, our findings suggest that the Armijo step-size strategy offers significant advantages in efficiency and speed, particularly in high-dimensional settings. 

We observe the performance of different step-size strategies (Armijo, adaptive, and diminishing) under varying values of \( \lambda \). Across all values of \( \lambda \) in Table~\ref{table:lambda_results}, the Armijo step size consistently outperforms the other methods in terms of both time and iterations. Specifically, at \( \lambda = 0.1 \), the Armijo method achieves the fastest convergence, taking an average of 0.3505 seconds and 53.7 iterations. As \( \lambda \) increases, the performance gap between Armijo and the other methods becomes more pronounced. For \( \lambda = 0.5 \), Armijo converges in 0.3153 seconds and 44.8 iterations, while the diminishing strategy requires 2.0290 seconds and 530.2 iterations, demonstrating its inefficiency for larger \( \lambda \) values. Similarly, the adaptive step size shows intermediate performance, but its time and iteration counts remain higher than Armijo. Overall, the results suggest that the Armijo step size is the most robust and efficient strategy for different values of \( \lambda \), while the diminishing method struggles with larger problem sizes and regularization terms.

\subsubsection{Comparison with the accelerated method}

We consider the same problem and compare the performance of the accelerated algorithm with the non-accelerated version. To better observe the convergence rate and differences between various step sizes, we run 50 iterations. As shown in Figure~\ref {fig:example_sphere_compare}, the results indicate that the accelerated algorithm with Armijo and diminishing step sizes generally outperforms the non-accelerated algorithm, particularly when compared with adaptive and diminishing step sizes. The results are similar to the method in \cite{li2021momentum}, and it can indeed accelerate the decline at the beginning.

Table \ref{table:sphere_compare} summarizes the results of 10 random runs for the Sphere Compare Experiment with \( \lambda = 0.1 \). Across all dimensions (\( n = 10, 100, 500, 1000 \)), the RGCG method combined with Armijo step size consistently achieves the best performance, with the lowest computation time and iteration count. For instance, with \( n = 1000 \), RGCG + Armijo takes 8.47 seconds and 346.8 iterations, outperforming the adaptive and diminishing strategies, which require more time and iterations.

The accelerated method shows similar trends, with Armijo step size yielding better results than adaptive and diminishing strategies. However, for higher dimensions, RGCG with Armijo step size remains faster overall. The diminishing step size, in both methods, leads to significantly longer runtimes and higher iteration counts, especially as \( n \) increases. 
For  $n = 1000, 2000$, indicating that the accelerated method benefits more from the diminishing step size in higher dimensions.
In summary, the Armijo step size proves to be the most efficient strategy, particularly when paired with the RGCG method, making it the preferred choice for both low and high-dimensional problems.

\begin{figure}[htbp]
    \centering
    \subfigure[Function value, $n=10$]{
        \label{Fig1.sub.1}
        \includegraphics[width=0.3\textwidth]{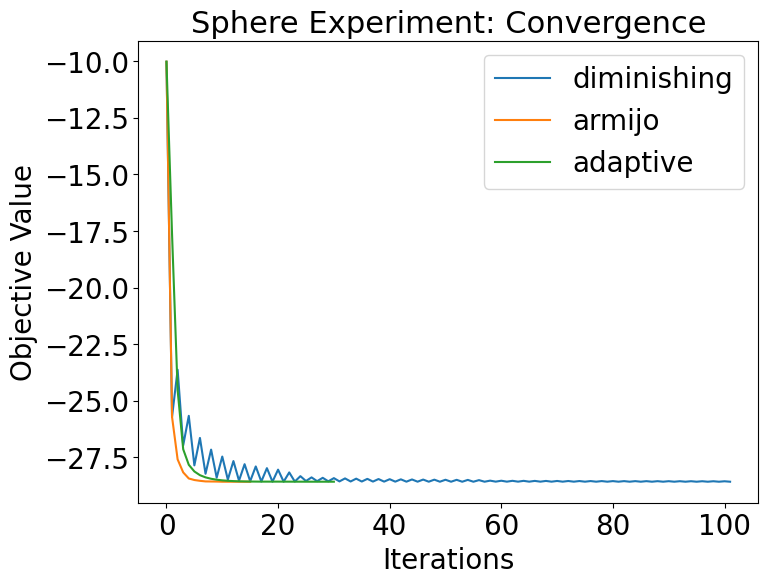}}
    \subfigure[Function value, $n=100$]{
        \label{Fig2.sub.1}
        \includegraphics[width=0.3\textwidth]{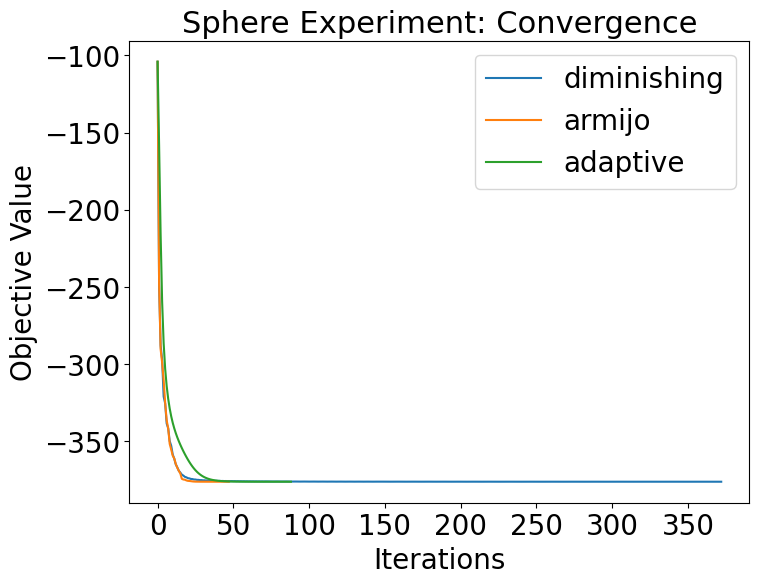}}
    \subfigure[Function value, $n=1000$]{
        \label{Fig3.sub.1}
        \includegraphics[width=0.3\textwidth]{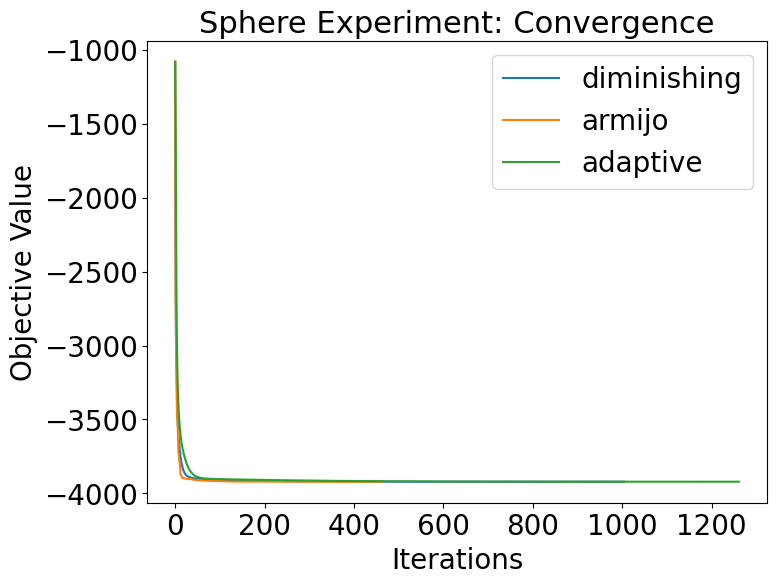}}
    
    \subfigure[Theta value, $n=10$]{
        \label{Fig1.sub.2}
        \includegraphics[width=0.3\textwidth]{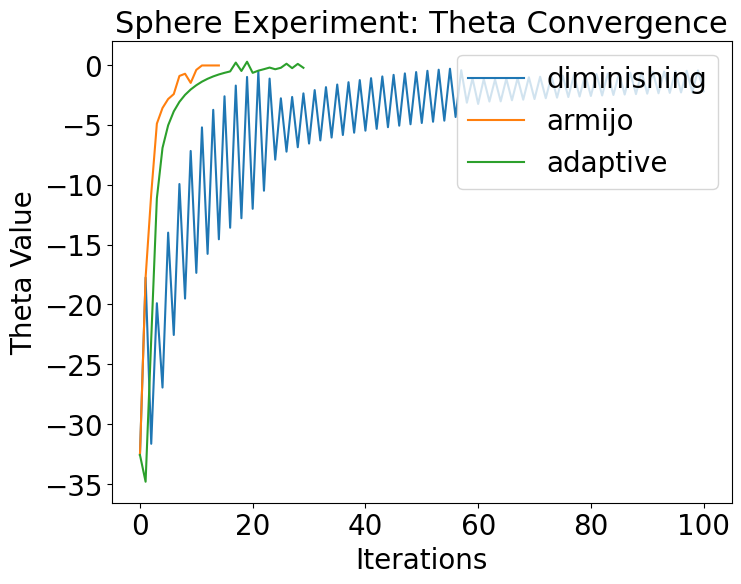}}
    \subfigure[Theta value, $n=100$]{
        \label{Fig2.sub.2}
        \includegraphics[width=0.3\textwidth]{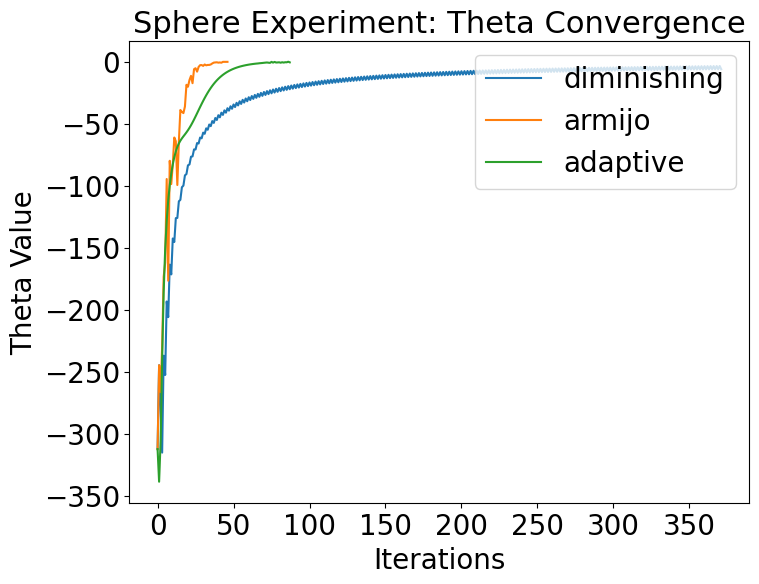}}
    \subfigure[Theta value, $n=1000$]{
        \label{Fig3.sub.2}
        \includegraphics[width=0.3\textwidth]{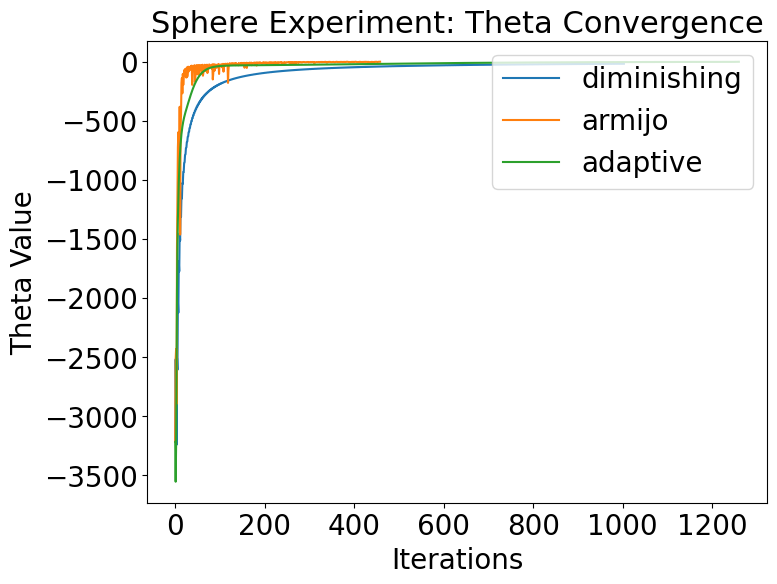}}
    
    \subfigure[CPU time, $n=10$]{
        \label{Fig1.sub.2}
        \includegraphics[width=0.3\textwidth]{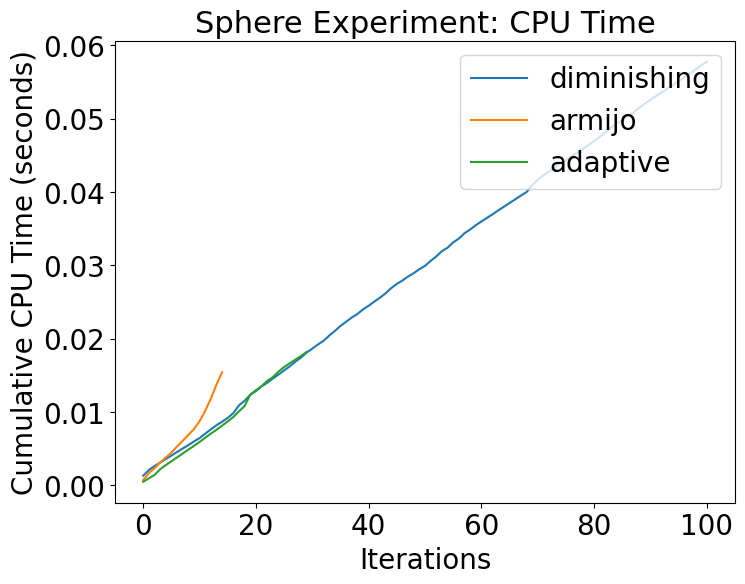}}
    \subfigure[CPU time, $n=100$]{
        \label{Fig2.sub.2}
        \includegraphics[width=0.3\textwidth]{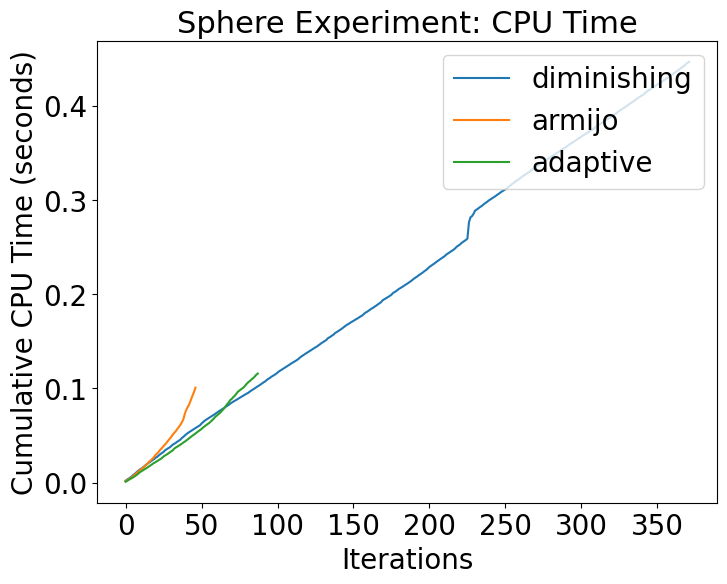}}
    \subfigure[CPU time, $n=1000$]{
        \label{Fig3.sub.2}
        \includegraphics[width=0.3\textwidth]{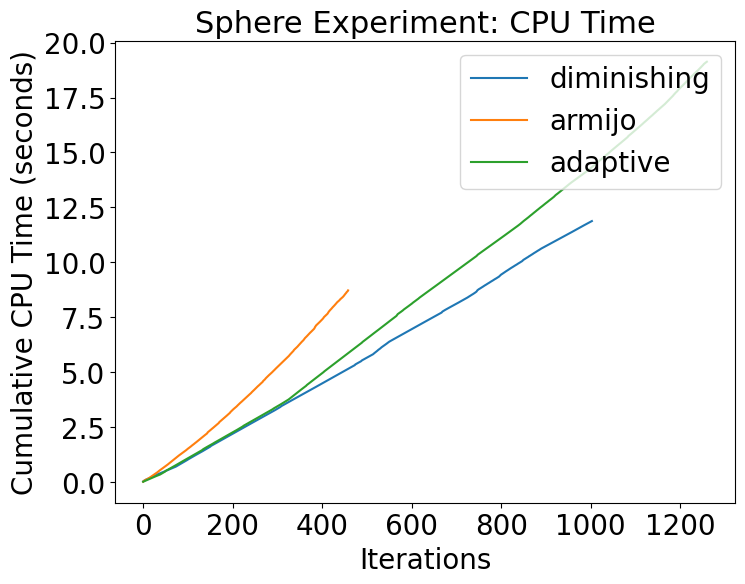}}
    
    \caption{Example \eqref{exam: sphere} with fixed $A$, different $n$ and \( \lambda = 0.1 \)}
    \label{fig:example_sphere}
\end{figure}

\begin{table}[h]
    \centering
    \caption{Average results of 10 random runs for problem \eqref{exam: sphere}  with \( \lambda = 0.1 \) and fixed $A$}
    \begin{tabular}{l p{2.5cm} p{2.5cm} p{2.5cm}}
        \toprule
          n & Step size & Time (seconds) & Iterations \\
        \midrule
        \multirow{3}{*}{10} 
          & Armijo & \textbf{0.0152} & \textbf{14.50} \\
          & Adaptive & 0.0208 & 31.30 \\
          & Diminishing & 0.0491 & 88.80 \\
        \midrule
        \multirow{3}{*}{100} 
          & Armijo & \textbf{0.0771} & \textbf{32.20} \\
          & Adaptive & 0.1761 & 78.70 \\
          & Diminishing & 0.3113 & 244.60 \\
        \midrule
        \multirow{3}{*}{1000} 
          & Armijo & \textbf{8.3842} & \textbf{344.60} \\
          & Adaptive & 17.4690 & 921.00 \\
          & Diminishing & 13.1272 & 929.70 \\
        \bottomrule
    \end{tabular}    
    \label{table 1}
\end{table}
\begin{table}[htbp]
    \centering
    \caption{Average results of 10 random runs for problem \eqref{exam: sphere} with varying \( \lambda \) and fixed $A$, \( n = 500 \)}
    \begin{tabular}{l p{2.5cm} p{2.5cm} p{2.5cm}}
        \toprule
        \( \lambda \) & Step size & Time (seconds) & Iterations \\
        \midrule
        \multirow{3}{*}{0.1} 
          & Armijo & \textbf{0.3505} & \textbf{53.70} \\
          & Adaptive & 0.7006 & 136.10 \\
          & Diminishing & 2.2043 & 522.90 \\
        \midrule
        \multirow{3}{*}{0.3} 
          & Armijo & \textbf{0.3987} & \textbf{51.60} \\
          & Adaptive & 0.6094 & 111.70 \\
          & Diminishing & 1.6226 & 381.80 \\
        \midrule
        \multirow{3}{*}{0.5} 
          & Armijo & \textbf{0.3153} & \textbf{44.80} \\
          & Adaptive & 0.5716 & 111.30 \\
          & Diminishing & 2.0290 & 530.20 \\
        \midrule
        \multirow{3}{*}{1.0} 
          & Armijo & \textbf{0.3016} & \textbf{46.50} \\
          & Adaptive & 0.5137 & 97.30 \\
          & Diminishing & 1.8412 & 465.50 \\
        \bottomrule
    \end{tabular}    
    \label{table:lambda_results}
\end{table}

\begin{figure}[htbp]
    \centering
    \subfigure[Function value, $n=10$]{
        \includegraphics[width=0.3\textwidth]{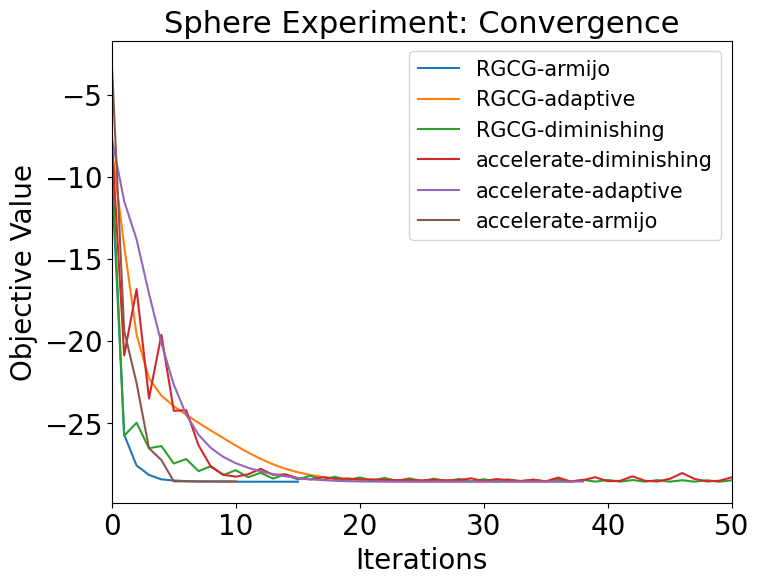}}
    \subfigure[Function value, $n=100$]{
        \includegraphics[width=0.3\textwidth]{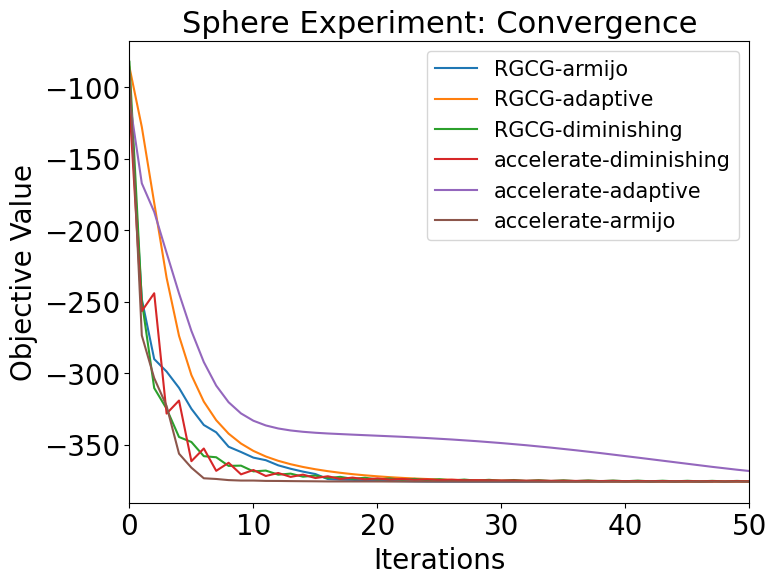}}
    \subfigure[Function value, $n=1000$]{
        \includegraphics[width=0.3\textwidth]{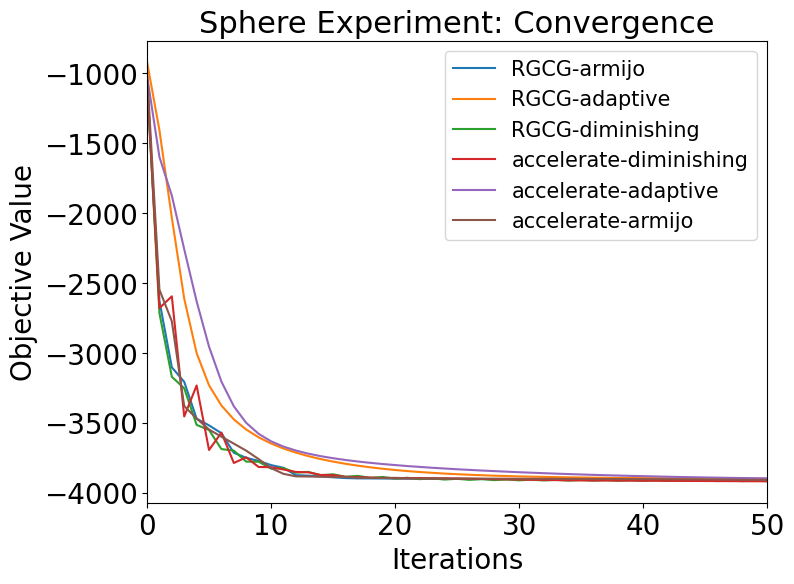}}
    
    \subfigure[Theta value, $n=10$]{
        \includegraphics[width=0.3\textwidth]{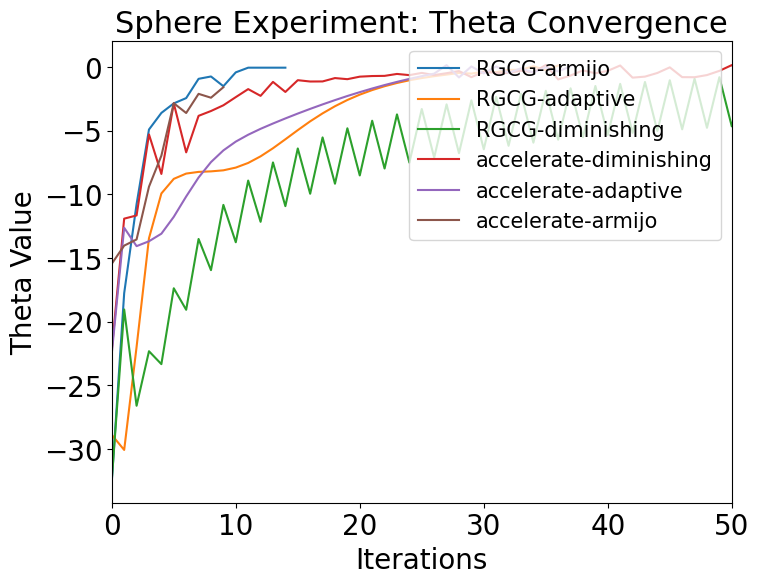}}
    \subfigure[Theta value, $n=100$]{
        \includegraphics[width=0.3\textwidth]{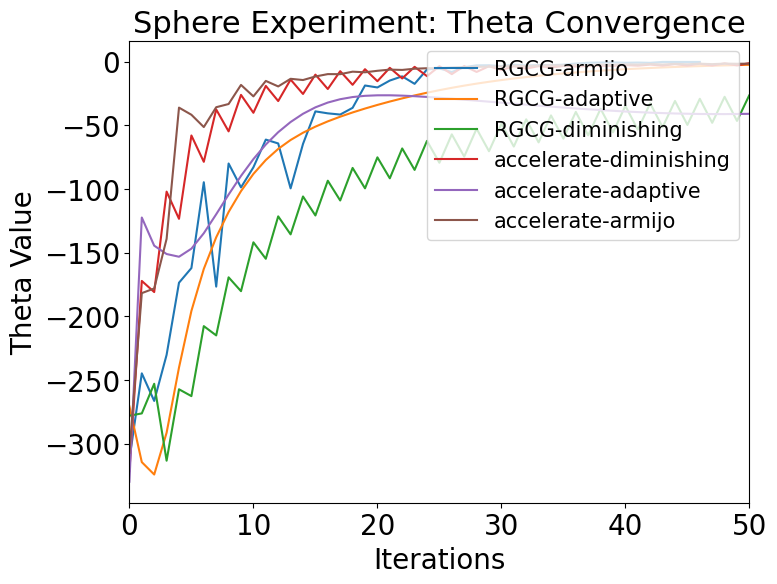}}
    \subfigure[Theta value, $n=1000$]{
        \includegraphics[width=0.3\textwidth]{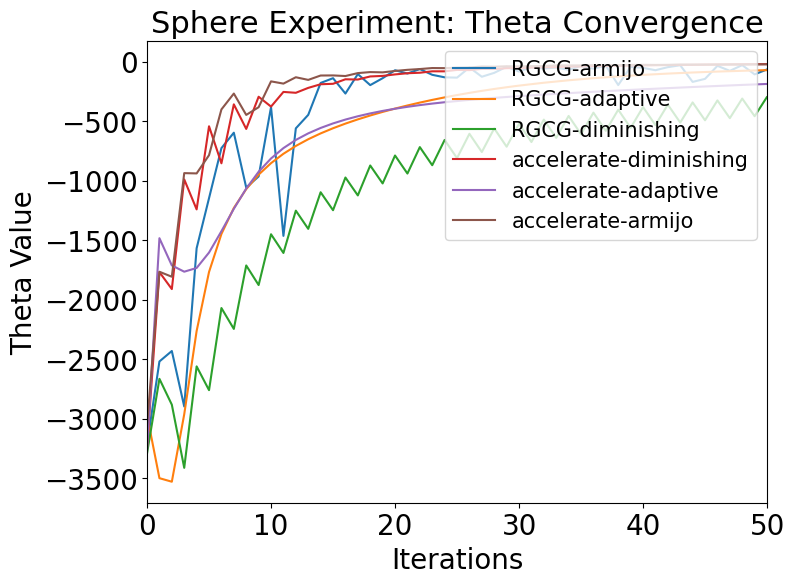}}
    
    \caption{Example \eqref{exam: sphere} with fixed $A$, \( \lambda = 0.1 \) and different $n$ for RGCG and accelerated method }
    \label{fig:example_sphere_compare}
\end{figure}

\begin{table}[htbp]
    \centering
    \caption{Comparison of Accelerated and RGCG Methods: Average Results of 10 Random Runs for problem \eqref{exam: sphere} with \( \lambda = 0.1 \)}
    \begin{tabular}{l p{4.5cm} p{2.5cm} p{2.5cm}}
        \toprule
          n & Method + Step size & Time (seconds) & Iterations \\
        \midrule
        \multirow{3}{*}{10} 
          & RGCG + Armijo & \textbf{0.0173} & \textbf{15.30} \\
          & RGCG + Adaptive & 0.0288 & 36.10 \\
          & RGCG + Diminishing & 0.0796 & 109.50 \\
        \midrule
        \multirow{3}{*}{10} 
          & Accelerated + Armijo & \textbf{0.0496} & \textbf{36.80} \\
          & Accelerated + Adaptive & 0.0507 & 61.30 \\
          & Accelerated + Diminishing & 0.1137 & 107.30 \\
        \midrule
        \multirow{3}{*}{100} 
          & RGCG + Armijo & \textbf{0.1069} & \textbf{32.40} \\
          & RGCG + Adaptive & 0.1670 & 83.00 \\
          & RGCG + Diminishing & 0.6486 & 295.50 \\
        \midrule
        \multirow{3}{*}{100} 
          & Accelerated + Armijo & \textbf{0.2629} & \textbf{85.30} \\
          & Accelerated + Adaptive & 0.3232 & 169.30 \\
          & Accelerated + Diminishing & 0.9453 & 368.50 \\
        \midrule
        \multirow{3}{*}{500} 
          & RGCG + Armijo & \textbf{0.3750} & \textbf{56.70} \\
          & RGCG + Adaptive & 0.8563 & 160.50 \\
          & RGCG + Diminishing & 2.4350 & 605.30 \\
        \midrule
        \multirow{3}{*}{500} 
          & Accelerated + Armijo & \textbf{1.3190} & \textbf{166.00} \\
          & Accelerated + Adaptive & 2.4168 & 385.60 \\
          & Accelerated + Diminishing & 2.8703 & 417.10 \\
        \midrule
        \multirow{3}{*}{1000} 
          & RGCG + Armijo & \textbf{8.4704} & \textbf{346.80} \\
          & RGCG + Adaptive & 17.8079 & 830.50 \\
          & RGCG + Diminishing & 14.0213 & 908.70 \\
        \midrule
        \multirow{3}{*}{1000} 
          & Accelerated + Armijo & \textbf{8.6065} & \textbf{255.30} \\
          & Accelerated + Adaptive & 13.2967 & 687.70 \\
          & Accelerated + Diminishing & 12.1202 & 445.60 \\
        \midrule
        \multirow{3}{*}{2000} 
          & RGCG + Armijo & \textbf{14.6191} & \textbf{160.20} \\
          & RGCG + Adaptive & 21.9078 & 411.40 \\
          & RGCG + Diminishing & 50.3453 & 1106.70 \\
        \midrule
        \multirow{3}{*}{2000} 
          & Accelerated + Armijo & \textbf{19.0932} & \textbf{188.00} \\
          & Accelerated + Adaptive & 47.8875 & 752.60 \\
          & Accelerated + Diminishing & 25.3371 & 396.20 \\
        \midrule
    \end{tabular}
    \label{table:sphere_compare}
\end{table}

\subsubsection{Stiefel experiment}
 For the example \eqref{ex:st}, we conduct experiments on a $(n,p)$-dimensional problem, using a randomly generated \( n \times n \) standardized symmetric \( A \) and a regularization parameter \( \lambda = 0.1 \). As in the previous experiment, here we evaluate Armijo step size with parameters \( \zeta = 0.1 \), \( \omega_1 = 0.05 \), and \( \omega_2 = 0.95 \). Each experiment runs for the same original problem 10 times with different initial points, and stops if the norm of \(\theta\) or the difference of the function value is less than \(10^{-4}\). The maximum number of iterations allowed for solving subproblems is set to 2.
 The performance of the RGCG method is analyzed by plotting the convergence behavior of the objective values and the value of \( \theta \) over iterations and the table of average performance.

The experimental results for \eqref{ex:st} on the Stiefel manifold (Table~\ref{experi_stiefel} and Figures~\ref{stfig:example_st}) show clear differences in performance between the three step-size strategies across various problem sizes. For smaller problem dimensions (e.g., \( n = 100 \) and \( n = 200 \)), the Armijo step size consistently achieves the fewest iterations, indicating its efficiency in reaching convergence faster. However, the diminishing step size, while stable, requires significantly more iterations and time, especially as the problem dimension increases.

For higher dimensions (\( n = 500 \) and \( n = 1000 \)), the adaptive step size shows a competitive balance between time and iterations, especially compared to the diminishing one, which becomes increasingly less efficient. However, Armijo continues to outperform both in terms of iterations, making it the preferred choice for high-dimensional problems where computational efficiency is critical.
Overall, the Armijo step size demonstrates superior performance in both time and iteration count, particularly for larger problem sizes. 
The diminishing step size, while robust, is less efficient for larger problems, indicating that it may not be suitable for high-dimensional optimization tasks.

We see a similar trend with the Armijo step size consistently performing well across varying values of \( \lambda \). At \( \lambda = 0.1 \), Armijo converges in 0.6839 seconds and 103.3 iterations, slightly slower than the adaptive method in terms of time (0.5516 seconds) but significantly faster than the diminishing strategy (1.3375 seconds and 409.9 iterations). For different \( \lambda \), Armijo maintains its competitive performance. For \( \lambda = 1.0 \), it achieves an average of 0.9943 seconds and 147.8 iterations, compared to 1.7331 seconds and 514.4 iterations for the diminishing step size. While the adaptive strategy generally outperforms the diminishing one, it remains less efficient than Armijo in terms of iterations. 

\begin{table}[htbp]
    \centering
    \caption{Average results of 10 random runs for problem \eqref{ex:st} with \( \lambda = 0.1 \) and fixed $A$}
    \begin{tabular}{l p{2.5cm} p{2.5cm} p{2.5cm}}
        \toprule
        (p, n) & Step size & Time (seconds) & Iterations \\
        \midrule
        \multirow{3}{*}{(10, 100)} 
          & Armijo & 0.2344 & \textbf{104.7} \\
          & Adaptive & 0.2287 & 233.0 \\
          & Diminishing & \textbf{0.1775} & 192.4 \\
        \midrule
        \multirow{3}{*}{(10, 200)} 
          & Armijo & \textbf{0.2729} & \textbf{59.2} \\
          & Adaptive & 0.4160 & 242.8 \\
          & Diminishing & 0.5557 & 347.0 \\
        \midrule
        \multirow{3}{*}{(10, 500)} 
          & Armijo & 1.8151 & \textbf{115.8} \\
          & Adaptive & \textbf{1.5056} & 229.5 \\
          & Diminishing & 5.5653 & 827.5 \\
        \midrule
        \multirow{3}{*}{(10, 1000)} 
          & Armijo & \textbf{11.9895} & \textbf{219.0} \\
          & Adaptive & 12.6799 & 459.3 \\
          & Diminishing & 44.2100 & 1630.5 \\
        \bottomrule
    \end{tabular}
    \label{experi_stiefel}
\end{table}
\begin{table}[htbp]
    \centering
    \caption{Average results of 10 random runs for problem \eqref{ex:st} with varying \( \lambda \) and fixed $A$, \( n = 200, p = 10 \)}
    \begin{tabular}{l p{2.5cm} p{2.5cm} p{2.5cm}}
        \toprule
        \( \lambda \) & Step size & Time (seconds) & Iterations \\
        \midrule
        \multirow{3}{*}{0.1} 
          & Armijo & 0.6839 & \textbf{103.3} \\
          & Adaptive & \textbf{0.5516} & 171.9 \\
          & Diminishing & 1.3375 & 409.9 \\
        \midrule
        \multirow{3}{*}{0.3} 
          & Armijo & 1.0916 & 166.6 \\
          & Adaptive & \textbf{0.6786} & 206.7 \\
          & Diminishing & 1.4409 & 442.2 \\
        \midrule
        \multirow{3}{*}{0.5} 
          & Armijo & 0.9410 & \textbf{136.0} \\
          & Adaptive & \textbf{0.5557} & 167.9 \\
          & Diminishing & 1.3503 & 402.8 \\
        \midrule
        \multirow{3}{*}{1.0} 
          & Armijo & 0.9943 & \textbf{147.8} \\
          & Adaptive & \textbf{0.5600} & 170.1 \\
          & Diminishing & 1.7331 & 514.4 \\
        \bottomrule
    \end{tabular}
    \label{experi_stiefel_lambda}
\end{table}

\begin{figure}[htbp]
    \centering
    \subfigure[Function value, $n=100, p=10$]{
        \label{stFig3.sub.1}
        \includegraphics[width=0.3\textwidth]{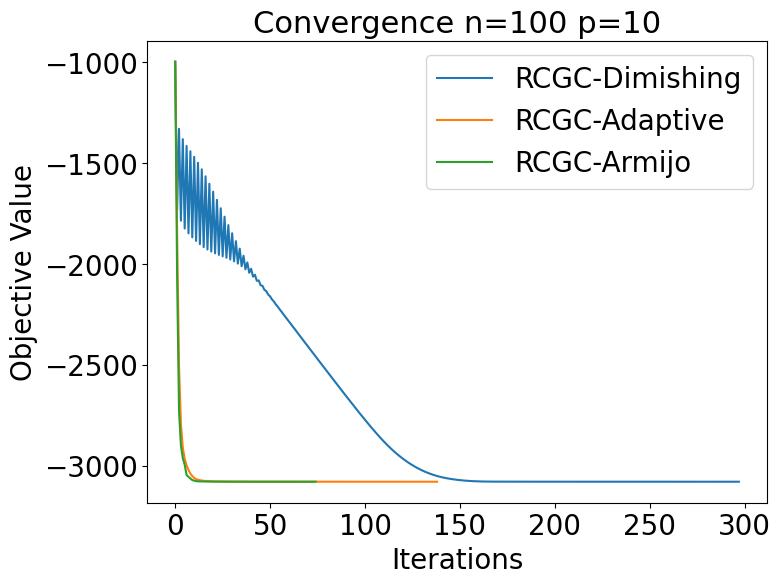}}
    \subfigure[Function value, $n=500, p=10$]{
        \label{stFig1.sub.1}
        \includegraphics[width=0.3\textwidth]{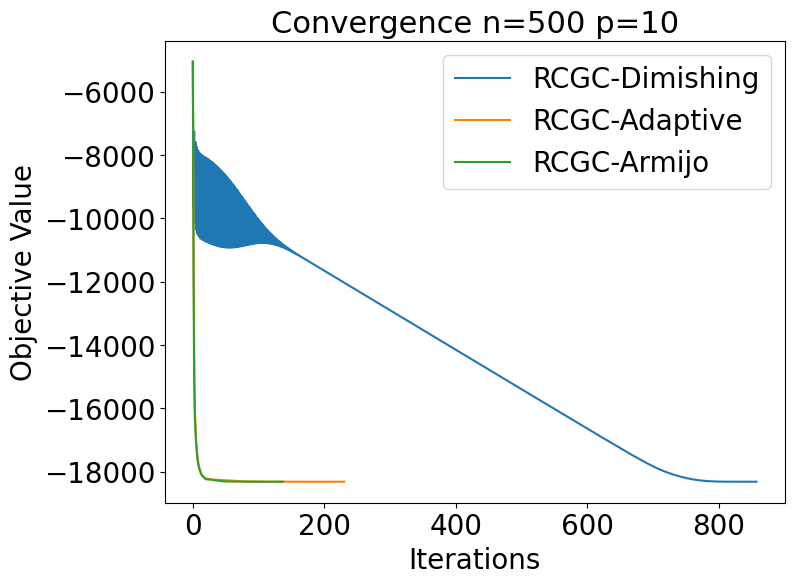}}
    \subfigure[Function value, $n=1000, p=10$]{
        \label{stFig2.sub.1}
        \includegraphics[width=0.3\textwidth]{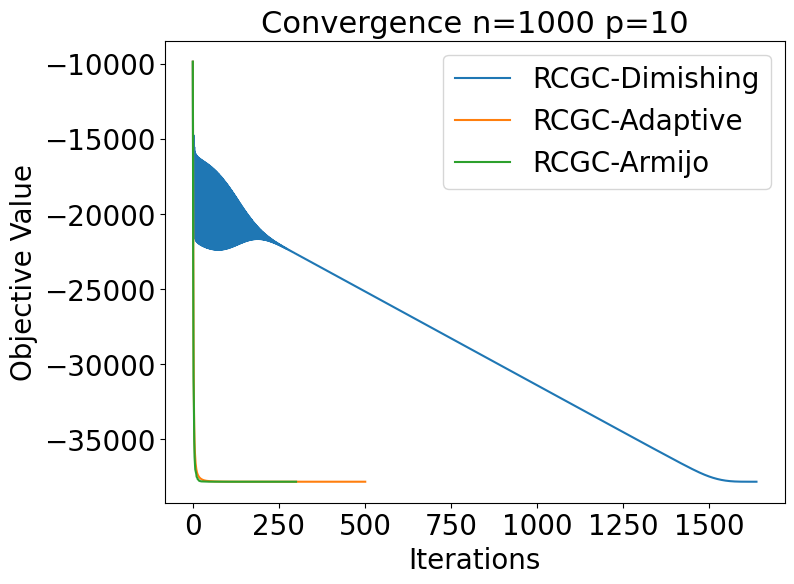}}
    
    \subfigure[Theta value, $n=100, p=10$]{
        \label{stFig3.sub.2}
        \includegraphics[width=0.3\textwidth]{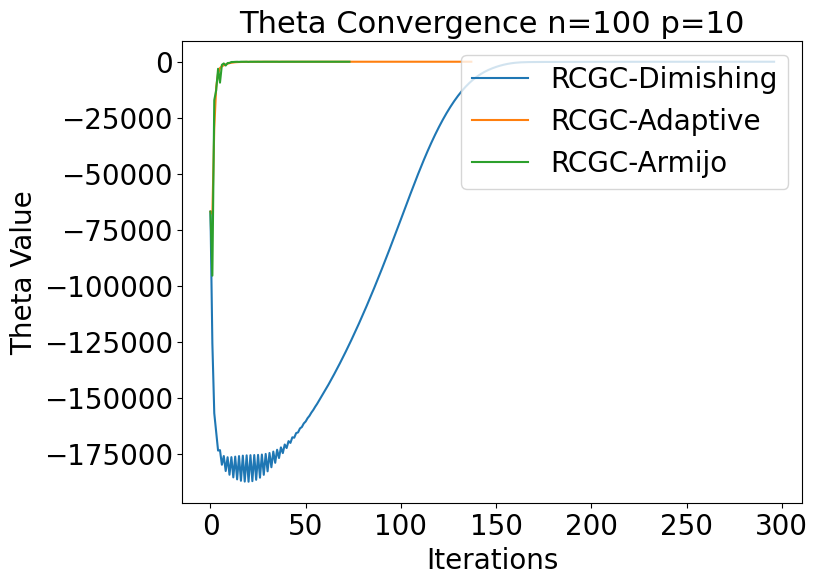}}
    \subfigure[Theta value, $n=500, p=10$]{
        \label{stFig1.sub.2}
        \includegraphics[width=0.3\textwidth]{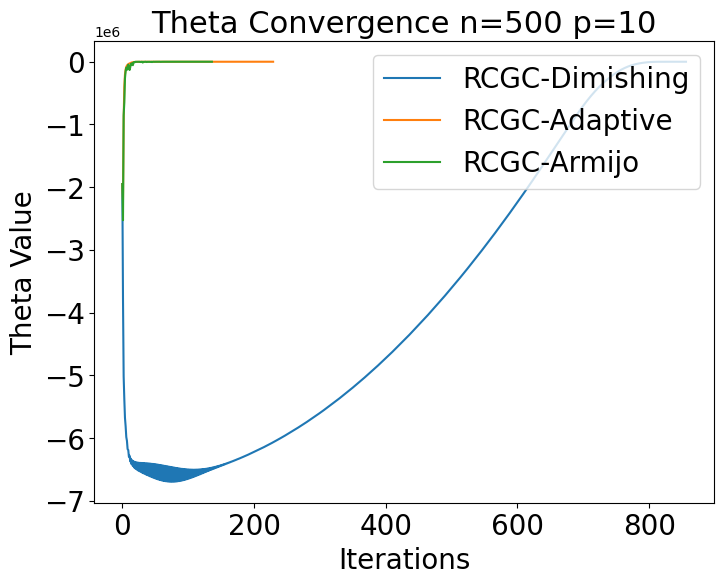}}
    \subfigure[Theta value, $n=1000, p=10$]{
        \label{stFig2.sub.2}
        \includegraphics[width=0.3\textwidth]{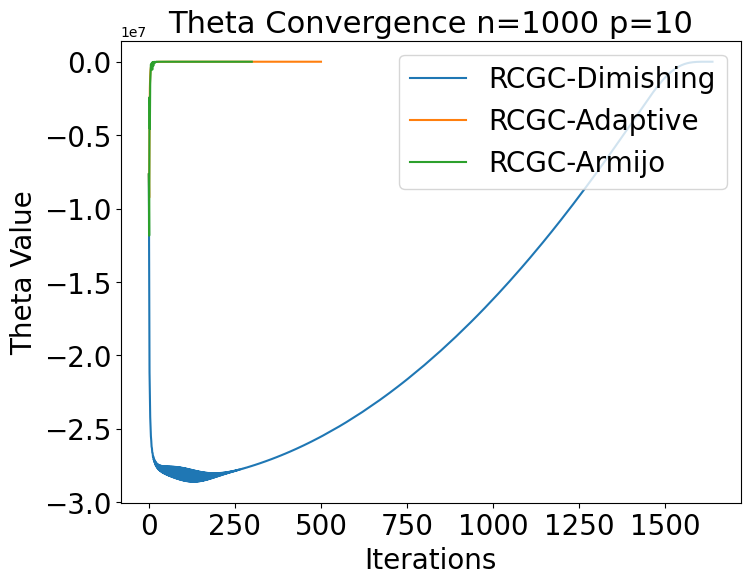}}

    \subfigure[CPU time, $n=100, p=10$]{
        \label{stFig3.sub.2}
        \includegraphics[width=0.3\textwidth]{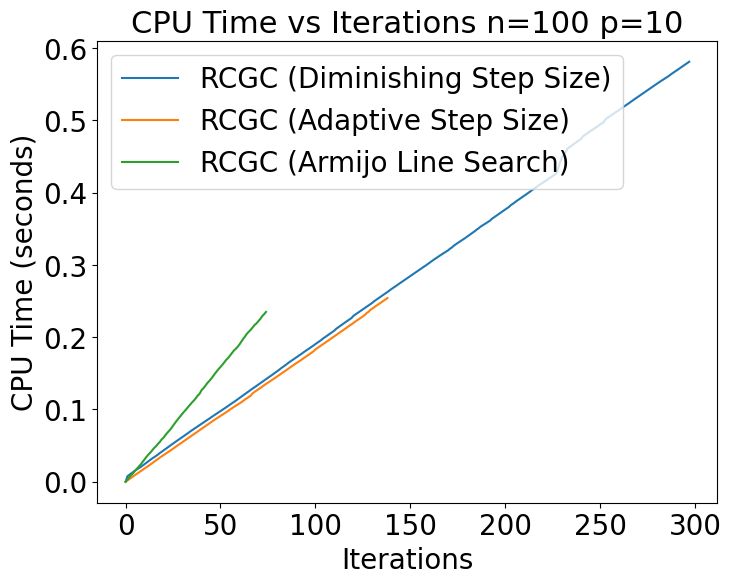}}
    \subfigure[CPU time, $n=500, p=10$]{
        \label{stFig1.sub.2}
        \includegraphics[width=0.3\textwidth]{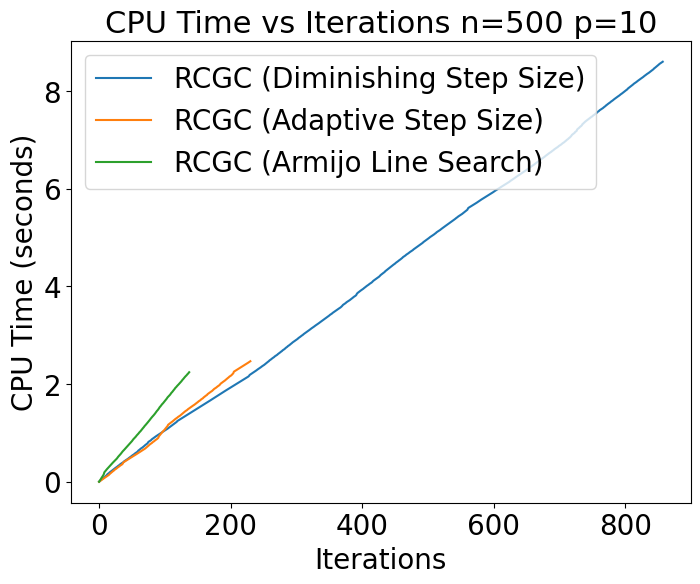}}
    \subfigure[CPU time, $n=1000, p=10$]{
        \label{stFig2.sub.2}
        \includegraphics[width=0.3\textwidth]{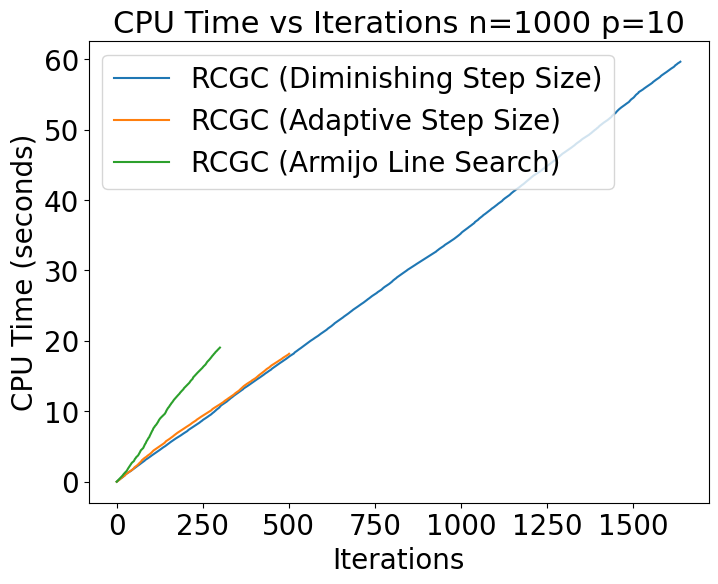}}

    \caption{Example \eqref{ex:st} with fixed $A$, \( \lambda = 0.1 \), different $n$ and $p$}
    \label{stfig:example_st}
\end{figure}

The overall results demonstrate that both the adaptive and Armijo step size methods outperform the diminishing step size in terms of convergence speed and computational time. Specifically, the adaptive strategy consistently exhibits the fastest convergence and the fewest iterations in most cases, while the Armijo step size method shows similar superiority in the majority of experiments. The performance of the algorithms is significantly influenced by different experimental parameters. Notably, as the number of rows and columns increases, the computational time and the number of iterations for all methods increase, but the adaptive and Armijo step sizes still perform better than the diminishing method. It was observed that the diminishing step size method experiences a significant slowdown in convergence speed as it approaches the stationary point. This indicates that while the method can initially reduce the objective function value quickly, its decreasing step size results in very small updates near the optimal solution, leading to a slower convergence rate. 
Although all step size strategies effectively converge for smaller problems, the choice of the appropriate step size becomes crucial for larger problems. These experiments suggest that in Stiefel manifold optimization, the adaptive and Armijo step size methods are generally more effective than the diminishing step size method, and their use is recommended in practical applications.

Through all the experiments, we observed that as the iteration point approaches the optimal solution, the descent slows down and requires more iterations to reach optimality. To address this issue, we might consider using a threshold or combining our approach with other methods to improve convergence.

\section{Conclusion and future works}
\label{sec: conclusion}
In this work, we proposed novel conditional gradient methods specifically designed for composite function optimization on Riemannian manifolds. Our investigation focused on Armijo, adaptive, and diminishing step-size strategies. We proved that the adaptive and diminishing step-size strategies achieve a convergence rate of \( \mathcal{O}(1/k) \), while the Armijo step-size strategy results in an iteration complexity of \( J + \mathcal{O}(1/\epsilon^2) \). In addition, under the assumption that the function \( g \) is Lipschitz continuous, we derived an iteration complexity of \( \mathcal{O}(1/\epsilon^2) \) for the Armijo step size.
Furthermore, we proposed a specialized algorithm to solve the subproblems in the presence of non-convexity, thereby broadening the applicability of our methods.
Our discussions and analyses are not confined to specific retraction and transport operations, enhancing the applicability of our methods to a broader range of manifold optimization problems. 

Also, from the results of the Stiefel manifold experiments, we observe that the convergence of the adaptive step size method tends to outperform the diminishing step size method. While the theoretical analysis previously focused on proving convergence rates through the diminishing step size, these results suggest that alternative step size strategies, such as adaptive methods, may lead to stronger performance. This opens the possibility of further exploration of different step size techniques to achieve better convergence rates in future studies.

Future work on manifold optimization could significantly benefit from the development and enhancement of accelerated conditional gradient methods. The momentum-guided Frank-Wolfe algorithm, as introduced by Li et al.~\cite{li2021momentum}, leverages Nesterov's accelerated gradient method to enhance convergence rates on constrained domains. This approach could be adapted to various manifold settings to exploit their geometric properties more effectively. Additionally, Zhang et al.~\cite{zhang2021accelerating} proposed accelerating the Frank-Wolfe algorithm through the use of weighted average gradients, presenting a promising direction for improving optimization efficiency on manifolds. Future research could focus on combining these acceleration techniques with Riemannian geometry principles to develop robust, scalable algorithms for complex manifold-constrained optimization problems. 
Another work for future research is to characterize the class of retractions that satisfy the Assumption~\ref{assumption convex}.
In addition to the exponential map, this assumption is also satisfied, for instance, if \( h \) is retraction-convex with respect to \( R \) in the sense of \cite[Definition 4.3.1]{huang2013optimization},
and \( y \) lies in a totally retractive neighborhood of $x$, where $R_x$ is a local diffeomorphism; the existence of such a neighborhood can be guaranteed by~\cite[Theorem 3.3.7]{flaherty2013riemannian}.

\section*{Acknowledgments} 
\noindent 
This work was supported by JST SPRING (JPMJSP2110), and Japan Society for the Promotion of Science, 
Grant-in-Aid for Scientific Research (C) (JP19K11840, JP25K15002).

\section*{Declarations} 
\textbf{Conflict of interest} 
Authors have no conflict of interest.


\bibliography{sn-bibliography}

\end{document}